\documentclass[12pt]{article}

\usepackage{amsmath,amsthm,amsfonts,graphicx,tikz-cd,pgfplots,listings}

\usepackage{geometry,titlesec,hyperref,xhfill,setspace,float,fancyhdr,cite}

\usetikzlibrary{positioning,calc}
\usepgfplotslibrary{polar}
\usepgflibrary{shapes.geometric}
\usepgfplotslibrary{fillbetween}

\pgfplotsset{my style/.append style={axis x line=middle, axis y line=
middle, xlabel={$x$}, ylabel={$y$}, axis equal }}

\theoremstyle{definition}
\newtheorem{definition}{Definition}[section]

\theoremstyle{plain}
\newtheorem{theorem}{Theorem}[section]
\newtheorem{lemma}{Lemma}[section]
\newtheorem{corollary}{Corollary}[section]

\DeclareMathOperator\Log{Log}

\newcommand{\underscore}{\underline{\hspace{2mm}}}
\newcommand{\Z}{\mathbb{Z}}

\newcommand{\C}{\mathbb{C}}
\newcommand{\R}{\mathbb{R}}
\newcommand{\HH}{\mathbb{H}}

\newcommand{\Aut}{\text{Aut}}
\newcommand{\SL}{\text{SL}}
\newcommand{\GL}{\text{GL}}
\newcommand{\SO}{\text{SO}}
\newcommand{\PGL}{\text{PGL}}
\newcommand{\PSL}{\text{PSL}}
\newcommand{\PSO}{\text{PSO}}

\newcommand{\Hom}{\text{Hom}}
\newcommand{\RP}{\mathbb{R}P}
\newcommand{\slf}{\mathfrak{sl}}
\newcommand{\so}{\mathfrak{so}}
\newcommand{\pgl}{\mathfrak{pgl}}
\newcommand{\g}{\mathfrak{g}}
\newcommand{\vv}{\mathfrak{v}}

\newcommand{\ad}{\text{ad}}
\newcommand{\Ad}{\text{Ad }}

\newcommand{\vbf}[1]{\textbf{#1}}
\newcommand{\lra}{\longrightarrow}
\newcommand{\wt}{\widetilde}
\newcommand{\tr}{\text{tr }}
\newcommand{\id}{\text{id}}

\newcommand{\im}{\text{im}}
\newcommand{\VD}{V^{*}}
\newcommand{\ZP}{\mathbb{Z}[\pi]}
\newcommand{\del}{\partial}
\newcommand{\res}{\text{res}}
\newcommand{\de}{\!:}
\title{Computer Assisted Projective Rigidity}
\author{Charles Daly}

\date{
	\today
}

\begin{document}
	\maketitle
	
	\begin{abstract}
		In this paper we provide a computer assisted proof that about two thousand surgeries far away from the ideal point in the hyperbolic Dehn filling space of the figure-eight knot complement are infinitesimally projectively rigid.  We also prove that for projective deformations of the figure-eight knot complement sufficiently close to the complete hyperbolic structure, the induced map on the first cohomology of the longitude of the boundary torus is non-zero.  This paper provides a complementary piece to the results of Heusener and Porti who showed that for each $k \in \Z$, there is a sufficiently large $N_{k}$ for which every $k/n$-Dehn filling on the figure-eight knot complement for $n \geq N_{k}$ is infinitesimally projectively rigid.  In the process of the proof, we provide explicit representations of the figure-eight knot complement in $PSO(3,1)$ which are rational in the real and imaginary parts of the shapes of the ideal tetrahedra used to glue the knot complement together.
	\end{abstract}

\begin{section}{Overview of the Paper}\label{secoverview}
It is a well known fact that the figure-eight knot complement, denoted here by $M$, supports a unique complete hyperbolic structure.  Using the thick-thin decomposition, one may cut the cusp of $M$ and is left with a compact manifold with toroidal boundary whose interior supports a hyperbolic structure.  Labeling a meridian and longitude of the toroidal boundary, one can glue in solid torus wrapping $p$-times around the meridian and $q$-times around the longitude to yield a \emph{closed} manifold.  What is remarkable is that for all but finitely many choices of relatively prime $(p,q)$, the resulting closed manifold inherits a hyperbolic structure.  This process is known as Thurston's $\emph{hyperbolic Dehn filling}$ and has been well-studied by many authors, \cite[Theorem 81.]{Agol2010Bounds}, \cite[Theorem 1.1]{Lackenby2012maximal}, \cite[Chapter 4]{Thurston2022Geometry}.  \\
\\
A recurring and very active question in the field of geometry is the following: given some geometric structure on a manifold, can one deform the structure to produce a new distinct geometry on the manifold?  Some examples include deforming the Euclidean structure on a flat torus or deforming hyperbolic structures on a genus $g$-closed orientable surface.  In the setting of closed orientable hyperbolic 3-manifolds, Mostow Rigidity says that no such non-trivial deformation exists.  This paper's primary concern is to provide evidence a similar phenomena occurs in the projective setting on manifolds that arise as Dehn surgeries on the figure-eight knot complement.  \\
\\
Consider a compact orientable hyperbolic 3-manifold $M$ with toroidal boundary.  Let $M_{p,q}$ denote the resulting manifold under $(p,q)$-Dehn surgery.  By Thurston's Hyperbolic Dehn filling Theorem \cite[Appendix B]{BoileauHeusenerPorti}, one obtains a discrete and faithful representation $\pi_{1}(M_{p,q}) \lra \PSL(2,\C) \simeq \PSO(3,1) \subset \PGL(4,\R)$, and thus the manifold inherits a natural projective structure.  Because of Mostow Rigidity, there is no non-trivial deformation of the hyperbolic structure inside the group $\PSO(3,1)$; however Mostow Rigidity is silent about deforming the projective structure in the larger group $\PGL(4,\R)$.  Cooper, Long, Thistlethwaite \cite{Cooper2006Computing} surveyed about four and a half thousand 2 generator hyperbolic 3-manifolds from the Hodgson-Weeks census and showed via a computer assisted proof that at most 61 of them admit projective deformations of the hyperbolic structure.  \\
\\
The work of Heusener and Porti \cite[Proposition 1.7]{Heusener2011Infinitesimal} shows that for each $k \in \Z$ one may find a sufficiently large $N_{k} > 0$ so that for all $n \geq N_{k}$, the $(k,n)$-Dehn filling on the figure-eight knot complement does not admit projective deformations.  However, in the same paper they did construct an \emph{orbifold} corresponding to $(n,0)$-Dehn surgery on the figure-eight knot complement $M$ that is \emph{not} projectively rigid and whose deformation space is a curve.  Passing to a cyclic cover, they obtain a \emph{manifold} that admits projective deformations; these are the so-called Fibonacci manifolds.  \\
\\
It appears that the general literature suggests that having non-trivial projective deformations of the hyperbolic structure on a closed manifold is quite rare.  Part of this paper is to add to this consensus regarding the projective rigidity of the figure-eight knot complement and provide conjectural evidence that in fact, \emph{all} such surgeries that result in an honest manifold on the figure-eight knot complement are projective rigid. \\
\\
The paper is organized as follows.  Section \ref{secprelim} is dedicated to reviewing some preliminaries about hyperbolic Dehn filling, the deformation space of a compact orientable hyperbolic 3-manifold with boundary, the Thurston-slice, and infinitesimal projective rigidity.  In addition this section reviews the definition of singular homology and cohomology with \emph{twisted coefficient} modules.  In this section, many of the standard constructions of ordinary cohomology and homology, e.g. Poincar\'{e} duality and the Kronecker pairing, are provided in the context of twisted coefficients.  Basic constructions of \emph{group cohomology} are also provided, though almost exclusively in dimensions zero and one.  \\
\\
In Section \ref{secPSOREP}, with the aid of Mathematica \cite{Mathematica}, we construct explicit representations of the figure-eight knot complement into $\PSO(3,1)$ that deform the complete hyperbolic structure.  This construction is done using an isometry between the upper half space model and the Klein model of hyperbolic 3-space.  These representations are novel insofar as they are rational functions of the real and imaginary parts of the shapes from the two ideal tetrahedra used glue the figure-eight knot complement together.  Here we also provide some illustrations of the deformation of the figure-eight knot complement constructed by Ballas in \cite[Section 5.1]{Ballas2014Deformations}.  We tile the Klein model by ideal tetrahedra and intersect the tiling with totally geodesic surfaces to illustrate the tiling and deformation.  \\
\\
Section \ref{secchmfig8} provides some minor strengthening of some theoretical results from Heusener and Porti \cite[Remark 1.8]{Heusener2011Infinitesimal} regarding the infinitesimal projective rigidity of the figure-eight complement.  Specifically, we use the fact that the figure-eight knot complement is a once-punctured torus bundle over the circle to rewrite the first cohomology of $M$ as the fixed point set of the monodromy map induced by the base circle acting on the first cohomology of the fiber.  Using this structure we provide a new proof, Theorem \ref{Fcohom}, that the figure-eight knot complement is infinitesimally projectively rigid.  We leverage this result to show in Corollary \ref{deforml} that for representations sufficiently close to the complete one, that the inclusion map from $\pi_{1}(l) \lra \pi_{1}(M)$ where $l$ is a longitude of boundary torus $\del M$, induces an injection on the level of cohomology, which is to say, if one deforms the projective structure of the figure-eight knot complement on $M$, one must also deform the projective structure on the longitude of the boundary torus.    
\\
\\
The latter half of Section \ref{secchmfig8} is dedicated to proving Theorem \ref{thmrigid} which provides a way of certifying that the $(p,q)$-Dehn filling of the figure-eight knot complement is infinitesimally projectively rigid.  This is done by constructing a normal form of a cohomology representative in the first cohomology of $M$ as in Section \ref{ssssc}.  Here we also introduce the notion of the \emph{slope} of a representation corresponding to the projective structure on the figure-eight knot complement.  Using the normal form of the cohomology representative and some hypotheses regarding the slope of the representation, we provide computationally verifiable sufficient conditions for Theorem \ref{thmrigid}.  \\
\\
Section \ref{ia} provides by a brief summary of \emph{interval arithmetic} and an explanation of how to use the interval arithmetic package INTLAB by Rump \cite{Ru99a} in Matlab \cite{MATLAB} combined with SnapPy \cite{SnapPy} by Culler, Dunfield, Goemer, and Weeks to certify that a particular $(p,q)$-Dehn filling is projectively rigid.  We then walk through the calculations with an example in the case of the $(2,3)$-Dehn surgery.  \\
\\
This paper makes frequent reference to the Mathematica file \texttt{PRFE\underscore Mathematica.nb} and is best read with it as an accompaniment to verify long algebraic calculations.  This file is accessible for download \cite{Daly}.  
\end{section}

\section*{Acknowledgements}
I wish to thank many folks for their contribution to this work.  Specifically, I want to thank Sam Ballas, Jeff Danciger, and Bill Goldman for their deep insights into the preliminaries of the field of deforming geometries and projective structures.  I also wishes to thank Javi G\'{o}mez-Serrano for his time explaining interval arithmetic and helping with some of the code and Tom Goodwillie for his time explaining twisted cohomology of fiber bundles to me.  I wish to extend a very deep thank you to both Joan Porti and Rich Schwartz, both of whom have contributed both moral and mathematical support throughout this paper.  Joan Porti has given me many hours explaining his work with Michael Heusener to me, corrected several mistakes of mine, and provided invaluable insight into his way of thought.  Rich Schwartz was the one who initially suggested this problem to me, proposed further study through interval arithmetic, and has provided me a wealth of patience and understanding for the past several years, for which, I am truly indebted.  Finally, I would like to dedicate this work to Lauren Daly who every single day helps make this work possible.

\begin{section}{Preliminaries}\label{secprelim}
This section is dedicated to providing the necessary tools and language in the field of projective deformations.  Much of these preliminaries are contained in Heusener Porti \cite[Sections 2 \& 3]{Heusener2011Infinitesimal}, though here we provide details intended for the less mature audience.  Experts in the area should consider this section auxiliary.  Topics such as Dehn filling, the Thurston slice, deformation space, infinitesimal projective rigidity, and twisted singular (co)homology are all covered.  For the sake of simplicity and the context of this paper, $N$ will typically denote a orientable 3-manifold without boundary, whereas $M$ will typically denote a compact orientable hyperbolic 3-manifold with toroidal boundary.  Frequently in this work, $M$ is the figure-eight knot complement.  
\begin{subsection}{Projective Rigidity}
Let $N$ be a 3-manifold without boundary.  A $\emph{hyperbolic}$ structure on $N$ is a maximal atlas of charts of $N$, $\phi_{i}: U_{i} \lra \HH^{3}$ into hyperbolic 3-space such that for each connected component $V \subset U_{i}\cap U_{j}$, there exists a hyperbolic isometry $g_{ij}\de \HH^{3} \lra \HH^{3}$ such that the restriction of $\phi_{i}\circ\phi_{j}^{-1}\de \phi_{j}(V) \lra \phi_{i}(V)$ is equal to $g_{ij}$.  If $M$ is a compact 3-manifold with boundary, we define a hyperbolic structure on $M$ to be one on the \emph{interior} of $M$.  A \emph{projective} structure on a manifold is defined similarly with charts into $\RP^{3}$ and where the coordinate changes are restrictions of projectivities in $\PGL(3,\R)$.  \\
\\
The geometry of a \emph{closed} hyperbolic 3-manifold $N$ is entirely determined by a discrete and faithful representation of its fundamental group into $\PSO(3,1)$; the projectivitization of the identity component of the group $SO(3,1)$ which preserves the Lorentzian inner product 
\begin{equation}\label{lorentz}
(x,x) := x_{1}^{2}+x_{2}^{2}+x_{3}^{2}-x_{4}^{2}
\end{equation}
This representation is called the \emph{holonomy representation}, and will be denoted by $\rho \de \pi_{1}(N) \lra \PSO(3,1) \subset \PGL(4,\R)$.  Here we are identifying $H^{3}$ with the Klein model, $K^{3}$, of hyperbolic 3-space and its group of isometries, $\PSO(3,1)$.  This embedding of $K^{3} \lra \RP^{3}$ as an open convex subset of $\RP^{3}$, which is equivariant with respect to the inclusion homomorphism $\PSO(3,1) \subset \PGL(4,\R)$, presents 3-dimensional hyperbolic geometry as a sub-geometry of 3-dimensional projective geometry.  In Heusener Porti \cite[Definition 1.2]{Heusener2011Infinitesimal}, they introduce the following definition.  
\begin{definition}\label{projrig}
A closed hyperbolic 3-manifold $N$ is called \emph{infinitesimally projectively} rigid if 
\begin{equation*}
H^{1}\left(\pi_{1}(N); \slf(4,\R)_{\Ad\rho}\right) = 0
\end{equation*}
\end{definition}
The above is the first group cohomology of $\pi_{1}(N)$ with coefficients in the module $\slf(4,\R)$ twisted by the composition of the holonomy representation of the projective structure of $N$, $\rho \de \pi_{1}(N) \lra \PSO(3,1)$, and the adjoint representation of $\PGL(4,\R)$, $\text{Ad} \de \PGL(4,\R) \lra \Aut(\slf(4,\R))$.  This is addressed more thoroughly in Section \ref{ssdehnfill}.  Here we are using the identification of the Lie algebra of $\slf(4,\R)$ and $\pgl(4,\R)$.  Because the universal cover of our hyperbolic manifolds is hyperbolic space, which is contractible, the abstract group cohomology is isomorphic to the topological cohomology with local coefficients \cite[Chapter VIII]{Brown1982Cohomology}.  Consequently we enjoy many of topological tools such as Poincaré Duality or Mayer-Vietoris.  By the work of Weil \cite[Theorem 1]{Weil1964Remarks} one may think of Definition \ref{projrig} as the statement that there are no non-trivial deformations of representation $\rho\de\pi_{1}(N) \lra \PGL(4,\R)$.  The fact that there are no non-trivial deformations of $\rho\de\pi_{1}(N) \lra \PSO(3,1)$ is a consequence of Mostow Rigidity, so we know for example that $H^{1}(\pi_{1}(N); \mathfrak{pso}(3,1)_{\Ad \rho}) = 0$.  \\
\\
They define projective rigidity for hyperbolic 3-manifolds with \emph{boundary} similarly.
\begin{definition}[Definition 1.3 \cite{Heusener2011Infinitesimal}]\label{projrigrel}
A compact orientable hyperbolic 3-manifold $M$ with toroidal boundary is called \emph{infinitesimally projectively rigid relative to the cusp(s)} if the inclusion map $i \de \del M\lra M$ induces an injection on the level of first cohomology with coefficients in the twisted $\pi_{1}(M)$-module $\slf(4,\R)$.  That is to say,
\begin{equation*}
H^{1}\left(\pi_{1}(M); \slf(4,\R)_{\Ad\rho}\right) \xrightarrow{i^{*}} H^{1}\left(\pi_{1}(\del M); \slf(4,\R)_{\Ad\rho}\right) 
\end{equation*}
is injective.  
\end{definition}
One way of interpreting this definition the following: non-trivial deformations of the projective structure of $M$ induce non-trivial deformations of the projective structure(s) on the toroidal boundary component of $\del M$.  This is to say, there is no way of deforming the geometry of $M$ without deforming the geometry of the boundary torus.  It is still a very active area of research to determine which compact orientable hyperbolic 3-manifolds are infinitesimally projectively rigid relative to their cusps.  In \cite[Theorem 1.1]{Ballas2018Convex}, Ballas, Danciger, and Lee leverage the condition of infinitesimal projective rigidity relative to their cusps to construct convex projective structures on the double of hyperbolic 3-manifolds.  In a forth coming paper by the same authors, they show for sufficiently large slopes within a sector of hyperbolic Dehn filling space about the complete point, the $(p,q$)-Dehn surgery is infinitesimally projectively rigid.  
\end{subsection}

\begin{subsection}{Dehn Filling}\label{ssdehnfill}
As mentioned in the introduction, a remarkable fact about cusped hyperbolic 3-manifolds is that one may use hyperbolic Dehn filling to obtain \emph{closed} hyperbolic 3-manifolds.  Let $M$ denote a compact orientable hyperbolic 3-manifold with a single cusp, and $\rho\de \pi_{1}(M) \lra \PSL(2,\C)$ denote the holonomy of $M$.  It should be mentioned the following process can easily be generalized to the multiple cusped-case, however for ease of explanation we consider the case where $M$ has only a single cusp.\\
\\
Topologically, the cusp of $M$ is $[0,\infty)\times T^{2}$, so in the truncated manifold one has $\pi_{1}(\del M) \simeq \Z^{2}$.  Take simple closed curves representing a meridian and longitude of this boundary torus and denote them $x, l \in \pi_{1}(\del M)$ respectively.  In his notes, Thurston introduced a single complex parameter $u \in \C$ defined in an open neighborhood $0 \in U$ inside the complex plane \cite[Chapter 4]{Thurston2022Geometry}.  This neighborhood parametrizes the holonomy representations $M$ with $0 \in U$ corresponding to the \emph{unique complete} structure on $M$.  Recall the hyperbolic structure on $M$ is complete if and only if the hyperbolic structure on the cusp is complete \cite[Theorem 4.10]{Purcell2020Hyperbolic}.  \\
\\
If we perturb the complete structure of $M$ to an incomplete one with holonomy $\sigma \de \pi_{1}(M) \lra \PSL(2,\C)$, we know that the structure on the cusp is incomplete.  In particular, the generators $\sigma(x)$ and $\sigma(l)$ are loxodromic.  Up to conjugation we may assume that 
\begin{equation}\label{bndtor}
\sigma(x) = \left(
\begin{array}{cc}
e^{u/2} & 0 \\
0 & e^{-u/2}
\end{array}
\right) \text{ and } \sigma(l) = \left(
\begin{array}{cc}
e^{v/2} & 0 \\
0 & e^{-v/2}
\end{array}
\right)
\end{equation}
where $u$ and $v$ are the complex lengths of the translations $\sigma(x)$ and $\sigma(l)$ respectively.  Note these quantities are only well defined up to a $\pm$ and if one writes them in Cartesian coordinates, they are given by a real number and an angle, the latter being well-defined up to integer multiples of $2\pi$.  This gives us a correspondence between representations $\sigma$ close to the complete structure on $M$, and a single complex parameter $u \in \C$ close to $0$.  We denote this correspondence by associating $\sigma$ to $\rho_{u}$, where $\rho_{0}$ is the complete structure on $M$.\\
\\
Note when $u\neq 0$, the equation in the variables $(p,q) \in \R^{2}$ 
\begin{equation*}
pu + qv = 2\pi i \text{ where }u \text{ and } v \text{ are the complex length of } \sigma(x) \text{ and }\sigma(l)
\end{equation*}
has a \emph{unique} solution.  We call the solution $(p,q)$ the \emph{generalized Dehn filling coefficients} of the representation $\rho_{u}$.  When $u = 0$, this corresponds to the complete structure on $M$ and we define the generalized Dehn filling coefficients of $u = 0$ to be $\infty \in S^{2} = \R^{2}\cup \{\infty\}$.  \\
\\
Thurston showed this correspondence defines a diffeomorphism between a sufficiently small neighborhood $U$ about $0 \in \C$ and a neighborhood of $\infty \in S^{2} = \R^{2} \cup \{\infty\}$.  The \emph{connected neighborhood of} $\infty \in S^{2}$ under this diffeomorphism is called \emph{hyperbolic Dehn filling space}; however it should be stated there does not appear to be a consensus in the literature as to what is exactly called \emph{hyperbolic Dehn filling space}.  To the author's knowledge the connected component of $\infty$ is also called the \emph{Thurston Slice}.  The connected component is \emph{not} to be confused with the subset of \emph{all} such $(p,q) \in S^{2}$ which support a hyperbolic structure.  The author has heard this distinct notion referred to as \emph{parameter space}.  Parameter space remains a mysterious object, and to the author's understanding questions about whether it is even path-connected are still not understood.  A detailed treatment of this space and more can be found in \cite[Section 6.3]{Purcell2020Hyperbolic} and \cite[Chapter 4]{Thurston2022Geometry}.  Typically for the rest of this paper $U \subset \C$ will denote a small open neighborhood about $0$ parametrizing the holonomy $\rho_{u}$ of the figure-eight knot complement, and $(p,q)$ will be the corresponding point in hyperbolic Dehn filling space associated to it.  It should also be noted that frequently the point $0 \in U$ or $\infty \in S^{2}$ is referred to as the \emph{ideal point} of the representation variety.  \\
\\
Let $u \in U$ be a point close to the complete structure on a compact orientable hyperbolic 3-manifold $M$ and let $(p,q)$ denote the corresponding parameters in hyperbolic Dehn filling space.  If $u \neq 0$ then $M$ is not complete; however, one may take its metric completion.  The case we are particularly interested in is the one where $(p,q)$ are relatively prime integers.  Thurston's Hyperbolic Dehn filling Theorem guarantees for all but finitely many of these pairs, the metric completion of $M$, denoted by $M_{p,q}$, will support a hyperbolic structure and correspond to the $(p,q)$-Dehn filling on $M$.  Specifically, if we label the boundary torus with $x$ and $l$ simple closed curves representing a meridian and longitude respectively, then $M_{p,q}$ topologically is obtained by gluing a solid torus $D^{2}\times S^{1}$ in such a way that the longitude of $D^{2}\times S^{1}$ is sent to the curve $x^{p}l^{q}$.  This provides us a means of generating many closed orientable hyperbolic 3-manifolds.  The $(p,q)$ for which $M$ does not support a hyperbolic structure are known as \emph{exceptional slopes}.  In the case where $M$ is the figure-eight knot complement it is known that the exceptional slopes are precisely $\{(1,0),(0,1),(1,1),(2,1),(3,1),(4,1)\}$ and all their reflections across the $x$ and $y$-axes \cite[Chapter 4]{Thurston2022Geometry}.  
\end{subsection}
\begin{subsection}{Singular Cohomology with Local Coefficients}\label{subseclocalcohom}
Let $\pi \!:= \pi_{1}(X)$ denote the fundamental group of a finite CW-complex $X$ and let $\rho \de \pi_{1}(X) \lra \GL(V)$ be a representation into a finite dimensional vector space $V$.  Let $\wt{X}$ denote the universal cover of $X$ and $\sigma: \Delta^{n} \lra \wt{X}$ denotes a singular $n$-simplex in $\wt{X}$.  Adopting the standard notation of Hatcher, \cite[Section 3.H]{Hatcher2002Algebraic}, we denote the free abelian group generated by all singular $n$-simplices in $\wt{X}$ by $C_{n}(\wt{X})$.  This $\Z$-module enjoys a \emph{left}-action of $\pi$ given by composition via deck transformations.  That is, if $\sigma: \Delta^{n} \lra \wt{X}$ denotes a singular $n$-simplex, then $\gamma\sigma$ denotes the composition $\gamma \sigma: \Delta^{n} \lra \wt{X}$ where $\gamma \in \pi$ acts on $\wt{X}$ by deck-transformations.  In this manner, $C_{n}(\wt{X})$ may be promoted to a \emph{left} $\ZP$-module.  Similarly, $V$ is a \emph{left} $\ZP$-module via by $\gamma v:= \rho(\gamma) v$ for each $\gamma \in \pi$.  \\
\\
We form the chain complex $C_{*}(X; V) := C_{*}(\wt{X})\otimes_{\ZP}V$.  Care is needed here, as we are forming the tensor product over the typically non-commutative ring $\ZP$.  Specifically, we instead consider $C_{*}(\wt{X})$ as a \emph{right} $\ZP$-module in this tensor product by declaring $\sigma \gamma := \gamma^{-1}\sigma$ where $\gamma^{-1}$ acts by the original left-action in the previous paragraph.  This chain complex $C_{*}(X;V)$ has the boundary map given by $\partial\otimes \id$, where $\del$ is the typical boundary map on singular $n$-chains and $\id$ is the identity on $V$.  The homology groups associated to the chain-complex $C_{*}(X;V)$ are denoted by $H_{*}(X; V)$.  We call these homology groups the \emph{homology of $X$ with coefficients in $V$ twisted by $\rho\de \pi_{1}(X) \lra V$}.  \\
\\
Similarly, we may form the dual chain complex of $C_{*}(X; V)$ by taking the $\ZP$-homomorphisms from $C_{*}(\wt{X})$ into $V$; here we return to viewing $C_{*}(\wt{X})$ as a \emph{left} $\ZP$-module.  Specifically, the cochain-complex is defined by $C^{*}(X;V) := \Hom_{\ZP}(C_{*}(\wt{X}),V)$ where the boundary map $\delta$ is defined in the following manner: for each $F \in \Hom_{\ZP}(C_{n}(\wt{X}),V)$, define $\delta F \in \Hom_{\ZP}(C_{n+1}(\wt{X}),V)$ via $(\delta F)(\sigma):= F(\del \sigma)$ where $\del$ is the typical boundary map on singular $n$-chains.  The cohomology groups of this cochain complex are denoted by $H^{*}(X;V)$ and defined to be the \emph{cohomology of $X$ with coefficients in $V$ twisted by $\rho \de \pi_{1}(X) \lra V$}.  
\end{subsection}
\begin{subsection}{Twisted Group Cohomology}\label{subsectwistedgroup}
Given a representation $\rho\de \pi \lra \GL(V)$ of an abstract group $\pi$ with $V$ a finite dimensional vector space, one may form the \emph{group cohomology of }$\pi$\emph{ with coefficients twisted by} $\rho$.  These cohomology groups will be denoted by $H^{*}(\pi; V_{\rho})$ where the subscript $\rho$ is suppressed in contexts where the representation is understood.  In our case, because we are concerned with manifolds $M$ whose universal covers are contractible, we have a natural isomorphism between the singular topological manifold cohomology of $M$ defined in Section \ref{subseclocalcohom} and the abstract group cohomology defined here.  See Brown \cite[Chapter 3]{Brown1982Cohomology} for more details on the general construction and L\"{o}h \cite{claraloh} for nicely detailed notes.  The construction in a sentence or two begins by taking a projective resolution $P_{*}$ of $\Z$ over the group ring $\ZP$ and considering the dual of the chain complex $P_{*} \otimes_{\ZP} V$.  An explicit construction can be obtained via the so-called \emph{bar resolution} of $\pi$.  \\
\\
The zero and one dimensional cases warrant the most attention.  Before defining the isomorphisms, we introduce the notion of a \emph{crossed-homorphism}.     
\begin{definition}\label{defcrossedmaps}
Let $\rho\de \pi \lra V$ be a representation of $\pi$.  A \emph{crossed-homomorphism} of $\pi$ for the representation $\rho$ is a function $f\de \pi \lra V$ that satisfies $f(ab) = f(a) + af(b)$ for all $a,b \in \pi$ where $a$ acts on $f(b)$ via $\rho$.  We say a crossed-homomorphism $f$ is a co-boundary if there exists a $v \in V$ for which $f(a) = (1-a)v$ for all $a \in \pi$.
\end{definition}
With these definitions we may state the following lemma which can be found in any introductory text about group cohomology, e.g \cite[Section IV.2]{Brown1982Cohomology}
\begin{lemma}\label{lemcrossediso}
Let $\rho\de \pi \lra V$ be a representation of $\pi$.  Then $H^{0}(\pi; V_{\rho})$ is isomorphic to the invariants of the representation, $V^{\rho(\pi)}$.  Moreover, $H^{1}(\pi; V_{\rho})$ is isomorphic to the quotient of the crossed-homomorphisms modulo the co-boundaries as defined in Definition \ref{defcrossedmaps}.  
\end{lemma}
Before stating the significance of the first cohomology group, we introduce another definition.
\begin{definition}\label{repdef}
Let $\rho: \pi \lra G$ be a representation of a group $\pi$ into a group $G$.  Let $\rho_{t}$ be a one-parameter family of representations of $\pi$ into $G$ for which $\rho_{0} = \rho$.  We say that $\rho_{t}$ is a \emph{deformation} of the representation $\rho$.  If there exists a path of elements $g_{t} \in G$ with $g_{0} = 1$ such that $\rho_{t} = g_{t}\rho g_{t}^{-1}$, we say that $\rho_{t}$ is a \emph{trivial deformation}.  If no such path of elements exists, we say that $\rho_{t}$ is a \emph{non-trivial deformation}.
\end{definition}
We emphasize the study of non-trivial deformations frequently corresponds to non-trivial deformations in the geometry of a manifold.  For example, in the context of compact orientable hyperbolic 3-manifolds, as stated in Section \ref{ssdehnfill}, there is a one-dimensional complex parameter $u \in U$ that varies the hyperbolic structure of $M$.  A simpler example comes from considering a Euclidean torus obtained as the quotient of $\R^{2}/\Z^{2}$ where $\Z^{2}$ is the standard integer lattice acting by translations.  One can change the Euclidean structure by extending the meridian or longitude to obtain non-equivalent Euclidean structures on $T^{2}$.  A simple invariant to distinguish them is area.  As one changes the lengths of these curves, the corresponding holonomy representations are non-conjugate, and thus, yield non-trivial deformations of the original holonomy.  This is in essence the Ehresmann-Thurston principle which loosely states that the there's a local homeomorphism between distinct $(G,X)$-structures and non-conjugate representations of the fundamental group of the manifold $M$ in question.  Thus understanding whether the geometry of $M$ can be deformed can in many instances be reduced to understanding if there exists non-trivial deformations of the holonomy representation.  This motivates the importance of the following theorem of Weil's Infinitesimal Rigidity Theorem \cite[Theorem 1]{Weil1964Remarks}.  
\begin{theorem}[Weil Rigidity]\label{thmweilrigidity}
Let $\rho \de\pi \lra G$ be a representation of $\pi$ into an algebraic Lie-group $G$.  Consider $\mathfrak{g}$ as a $\pi$-module through composing $\rho$ with the adjoint representation of $G$.  Denote the resulting $\pi$-module by $\mathfrak{g}_{\Ad \rho}$.  If $H^{1}(\pi; \mathfrak{g}_{\Ad \rho}) = 0$ then $\rho$ admits no non-trivial deformations.  
\end{theorem}
It is the goal of this paper to employ Theorem \ref{thmweilrigidity} to show that many of the Dehn-surgered figure-eight knot complement manifolds are rigid in $\PGL(4,\R)$, and thus admit no deformations of the projective structure inherited from the complete hyperbolic one.    \\
\\
It is worth mentioning that while this condition is sufficient for infinitesimal rigidity, it is not necessary.  The work of Kapovich and Millson in their `Murphy's Law' paper shows that typically if one has a non-zero element in $H^{1}(\pi; V)$, the cocycle typically does not integrate to a deformation of the original representation of $\rho: \pi \lra \GL(V)$ \cite[Theorem 10.7]{Kapovich1996On}.  Typically, there are an infinite number of obstructions living inside $H^{2}(\pi; V)$ that need to be satisfied.  
\end{subsection}
\begin{subsection}{(Co)homology Pairings}\label{cohompairings}
As in Section \ref{subseclocalcohom}, let $X$ be a finite CW-complex and consider a representation of $\pi_{1}(X)$ into a finite dimensional vector space $V$.  Much like in the case of standard (co)homology theory, we get a well defined bilinear pairing between elements of $H^{p}(X; \VD)$ and $H_{p}(X; V)$ simply by evaluation.  We denote the pairing 
\begin{equation*}\label{kroneckerpair}
\left< \cdot \, , \, \cdot \right>\de H^{p}(X; \VD) \otimes H_{p}(X; V) \lra \R
\end{equation*}
This pairing is induced by the map taking an element $(F, c_{p}\otimes v) \in C^{p}(X; \VD) \times C_{p}(X; V)$ and evaluating $(F(c_{p}))(v)$.  One may readily check that this pairing descends to the level of (co)homology by showing the boundary operators are respected.  That this pairing is non-degenerate is a consequence that the boundary operators on the chain complexes  $C^{p}(X; \VD)$ and $C_{p}(X; V)$ are dual to one another.  Similar to the standard theory, we call this pairing the \emph{Kronecker pairing}.  Note this provides a natural isomorphism between $H_{p}(X; V)^{*}$ and $H^{p}(X; \VD)$ where the dual of (co)homology groups is as $\R$-modules.  This duality is easily remembered as the dual of homology in $V$ is the cohomology in the dual of $V$.\\
\\
In a manner similar to the standard theory, for compact oriented $n$-manifolds $M$ with boundary $\del M$, one may obtain a non-degenerate pairing between 
\begin{equation*}\label{pdpair}
H_{p}(M; \VD) \otimes H_{n-p}(M, \del M; V ) \lra \R
\end{equation*}
via Poincar\'{e}-Lefschetz duality, see \cite{Steenrod1943Homology} for a detailed construction.  This pairing induces an isomorphism between $H_{p}(M; \VD)^{*}$ and $H_{n-p}(M, \del M; V)$.  Because the dual of homology in $W$ is isomorphic to cohomology in the dual of $W$ with $W = V^{*}$, we obtain an isomorphism between $H^{p}(M; V)$ and $H_{n-p}(M, \del M; V)$.  This Poincar\'{e} duality map is denoted by 
\begin{equation*}\label{PD}
PD: H_{n-p}(M, \del M; V) \lra H^{p}(M; V)
\end{equation*}
We may construct another perfect pairing using the cup-product.  Given a cochain $z \in H^{n-p}(M, \del M; V)$ and another $w \in H^{p}(M; \VD)$, the consider the cup product 
\begin{equation*}\label{cupprod}
z\cup w \in H^{n}(M;\del M; V\otimes \VD)
\end{equation*}
We compose on the module level with $V\otimes \VD \lra \R$ with the usual perfect pairing between vectors and covectors, i.e. the trace, to obtain the non-degenerate pairing 
\begin{equation}\label{cuppair}
\cup\de H^{n-p}(M, \del M; V) \otimes H^{p}(M; \VD) \lra \R
\end{equation} 
It is worth emphasizing up until now, all these pairings have been defined solely in terms of topology and not on the existence of a $\pi_{1}(M) := \pi$-invariant correspondence between $V$ and $\VD$.  Recall such correspondences are in bijection with $\pi$-invariant non-degenerate bilinear forms on $V$.  In the case where we have such a $\pi$-invariant non-degenerate bilinear form $b \de V\times V \lra \R$, we have an identification between $V$ and $\VD$ as $\pi$-modules.  In turn, this yields an isomorphism between $H^{p}(M; V)$ and $H^{p}(M; \VD)$.  Applying this isomorphism to Equation \ref{cuppair}, our bilinear form $b$ induces a non-degenerate pairing between $H^{n-p}(M; V)$ and $H^{p}(M; V)$.  We abusively also denote this pairing by $\cup$.   
\begin{equation*}\label{bindcuppair}
\cup\de H^{n-p}(M, \del M; V) \otimes H^{p}(M; V) \lra \R
\end{equation*} 
In the context of this paper, the existence of a $\pi$-invariant non-degenerate bilinear form on $V$ will be guaranteed by the \emph{Killing form} through Cartan's Criterion \cite[Theorem 1.42]{Knapp1996Lie}.  Specifically, let $\g$ denote the Lie-algebra of a semi-simple Lie group $G$.  For any $X, Y \in \g$ denote the Killing-form by $B(X,Y) := \text{tr}(\ad_{X}\circ\ad_{Y})$ where $\ad_{X}$ is the adjoint operator on the Lie-algebra $\g$.  In the context of $\g = \slf(4,\R)$, up to scale, $B(X,Y) = 8\text{tr}(XY)$.  \\
\\
Recall that $\SO(3,1)$ is the group of unit determinant matrices which preserve the Lorentzian inner product defined in Equation \ref{lorentz}.  Through the standard orthonormal basis, we may represent the Lorentzian inner product by the matrix
\begin{equation*}\label{lorentzmat}
J = \left(
\begin{array}{cccc}
1 & 		&		&		\\
 & 		1&		&		\\
 & 		&		1&		\\
 & 		&		&	-1	
\end{array}
\right)
\end{equation*}
\\
so that $\SO(3,1) = \{A \in \GL(4,\R) \, | \, \det(A) = 1 \text{ and } A^{T}JA = J\}$.  One may readily verify that the corresponding Lie-algebra is given by
\begin{equation*}
\mathfrak{so}(3,1) = \{a \in \mathfrak{gl}(4,\R) \, | \, \text{tr}(a) = 0 \text{ and } a^{T}J = -Ja\}
\end{equation*}
where $a^{T}$ denotes the transpose of $a \in \mathfrak{gl}(4,\R)$.  This is the Lie-algebra of the connected component of $\SO(3,1)$, namely, the component that preserves the upper sheet of the hyperboloid defined by $(x,x) = -1$.  
\end{subsection}
\end{section}

\begin{section}{$PSO(3,1)$-representation}\label{secPSOREP}
Let $M$ denote the figure-eight knot complement.  In this section we furnish a family of representations of $\pi_{1}(M)$ into $\PSO(3,1)$ analogous to the deformations of the discrete and faithful representation of $\pi_{1}(M)$ in $\PSL(2,\C)$.  We do this by picking `nice' points in the Klein model of hyperbolic 3-space in such a way that the coordinates of the points, and consequently the matrix representations the transformations in $\PSO(3,1)$, are rational functions in the real and imaginary parts of the shapes of the ideal tetrahedra used to glue together the figure-eight knot complement.  With these representations established, we also review a splitting of the $\pi_{1}(M)$-module $\slf(4,\R)$ into two pieces where one piece will correspond to the projective deformations.  We conclude the section with illustrations depicting the projective deformation of $M$ as constructed by Ballas in \cite{Ballas2014Deformations}.  A remarkable aspect his construction is that the boundary of the invariant convex body is not strictly convex, i.e. it contains non-trivial line segments \cite[Theorem 1.1]{Ballas2015Finite}.  
\begin{subsection}{Klein Model}\label{kleinmodel}
In this section we recall some elementary facts about the Klein model for hyperbolic 3-space, $K^{3}$.  Denote $\R^{3,1}$ as $\R^{4}$ equipped with the Lorentz metric $(\vbf{x},\vbf{y}) := x_{1}y_{1}+x_{2}y_{2}+x_{3}y_{3}-x_{4}y_{4}$ for vectors $\vbf{x} = (x,y,z,w) \in \R^{4}$.  Define $K^{3}$ to be the projectivization of $||\vbf{x}||^{2} < 0$ under the standard projectivization map $\R^{4}\setminus 0 \lra \RP^{3}$.  We may choose coordinates on $K^{3}$ by considering the chart induced by the hyperplane $w = 1$.  Specifically, the coordinates $(x,y,z) \in \R^{3}$ correspond to the point $[x:y:z:1] \in K^{3}$.   \\
\\
Under this identification, we have an isometry from $K^{3}$ to the upper half space model, $H^{3}$ as given in \cite[Section 7]{Levy1997Flavors}. 
\begin{equation*}\label{isom1}
f(x,y,z) = \left(\frac{2x}{1+z},\frac{2y}{1+z},\frac{2\sqrt{1-x^{2}-y^{2}-z^{2}}}{1+z}\right)
\end{equation*}
Its inverse $g$ from $H^{3}$ to $K^{3}$ is given by 
\begin{equation}\label{isom2}
g(x,y,z) = \left(\frac{4x}{4+x^{2}+y^{2}+z^{2}},\frac{4y}{4+x^{2}+y^{2}+z^{2}},\frac{4-x^{2}-y^{2}-z^{2}}{4+x^{2}+y^{2}+z^{2}}\right)
\end{equation}
These isometries extend to the boundaries of their respectively models, $S^{2}$ for $K^{3}$ and $\hat{\C}$ for $H^{3}$, and we observe for example that $g$ takes $0 \in \del H^{3}$ to the point $(0,0,1) \in \del K^{3}$ whereas $\infty \in \del H^{3}$ gets sent to $(0,0,-1) \in \del K^{3}$.\\
\\
To obtain our representation of the figure-eight knot complement group into $\PSL(2,\C)$, we consider the gluing diagram in Figure \ref{figgluingdiag}.
\begin{figure}
	\begin{center}
	\includegraphics[scale=0.5]{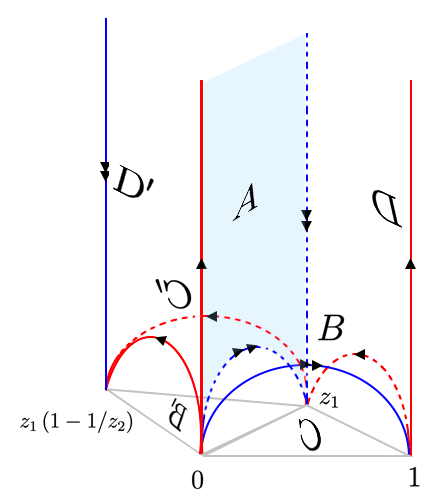}
	\end{center}
	\caption{Gluing diagram for the figure-eight knot complement for choices of complex numbers $z_{1}$ and $z_{2}$ in $\del H^{3}$ with positive imaginary parts.}\label{figgluingdiag}
\end{figure}
There you will see two ideal tetrahedra glued together along their common face $A = A'$.  The common face is depicted in light-blue.  The remaining unprimed faces and primed faces are glued together via the unique hyperbolic isometries taking them to each other.  The way we have chosen the shapes of the ideal tetrahedra are to align with SnapPy's triangulation of the figure-eight knot complement which we will describe here in detail.   \\
\\
In SnapPy, If we let \texttt{M = Triangulation('m004')} and enter \texttt{M.gluing\textunderscore equations()}, we see the following $4\times 6$-matrix produced. 
\begin{equation}\label{SNAPPYgluing}
\left[\begin{array}{ccccccc}
 2 & 1 & 0 & 1 & 0 & 2\\
  0 & 1 & 2 & 1 & 2 & 0\\
   1 & 0 & 0 & 0 & -1 & 0 \\
    0 & 0 & 0 & 0 & -2 & 2
\end{array}\right]
\end{equation}  
Each row corresponds to an equation in terms of the shape parameters $z_{1}$ and $z_{2}$ as in Figure \ref{figgluingdiag}.  Specifically, Row 1 of Equation \ref{SNAPPYgluing} corresponds to the equation
\begin{align}\label{snapglue1}
\underline{2}\cdot \Log z_{1} &+ \underline{1}\cdot \Log\left(\frac{1}{1-z_{1}}\right) + \underline{0} \cdot \Log\left(1-\frac{1}{z_{1}}\right) \nonumber \\
&+  \underline{1}\cdot\Log z_{2} + \underline{0}\cdot \Log\left(\frac{1}{1-z_{2}}\right) + \underline{2}\cdot\Log\left(1-\frac{1}{z_{2}}\right) = 2\pi i 
\end{align}
Similarly, Row 2 of Equation \ref{SNAPPYgluing} corresponds to the equation
\begin{align}\label{snapglue2}
\underline{0}\cdot \Log z_{1} &+ \underline{1}\cdot \Log\left(\frac{1}{1-z_{1}}\right) + \underline{2} \cdot \Log\left(1-\frac{1}{z_{1}}\right) \nonumber \\
&+  \underline{1}\cdot\Log z_{2} + \underline{2}\cdot \Log\left(\frac{1}{1-z_{2}}\right) + \underline{0}\cdot\Log\left(1-\frac{1}{z_{2}}\right) = 2\pi i
\end{align}
Rows 3 and 4 of Equation \ref{SNAPPYgluing} correspond to \emph{completeness equations} regarding the cusp of the figure-eight knot complement.  As we are not requiring our hyperbolic structure to be complete, we may ignore these equations.  It is worth recalling that SnapPy uses the following labeling scheme for its ideal tetrahedra. 
\begin{figure}
	\begin{center}
	\includegraphics[scale=0.5]{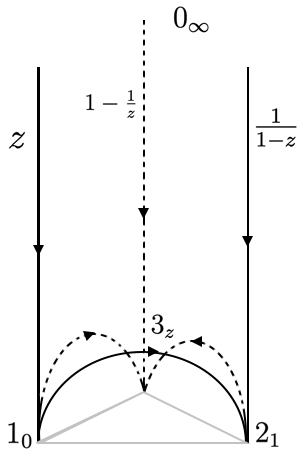}
	\end{center}
	\caption{Typical oriented tetrahedron in standard position $[\infty,0,1,z]$ decorated with some of the edge-invariants}\label{typtet}
\end{figure}
In Figure \ref{typtet}, you can see an ideal tetrahedron $[0,1,2,3]$ whose coordinates in $\del H^{3}$ are $\infty, 0, 1, z$ respectively.  They are decorated as subscripts to help the reader.  The edge invariants associated to the geodesics through $0$, whose coordinate is $\infty$, are labeled.  Recall opposite edges have the same edge invariant, however to keep Figure \ref{typtet} uncrowded, the opposite labels have been suppressed.  These edge-invariants are used to calculate the gluing equations associated to a triangulation of a hyperbolic 3-manifold.  More details about gluing equations from SnapPy can be found in \cite{SNAPPYGluingEQs}.  \\
\\
One may recover Equations \ref{snapglue1} and \ref{snapglue2} from Figure \ref{figgluingdiag} by inspecting the corresponding Thurston gluing equations.  Reading these gluing equations is a relatively involved process that requires one to align each ideal tetrahedra in a manner similar to Figure \ref{typtet}.  More details can be found in \cite[Chapter 4]{Thurston2022Geometry}, but we loosely explain the process here.  \\
\\
For each edge in a gluing diagram, we assign an equation in the shapes of the ideal tetrahedra used to glue the manifold together.  Referring to Figure \ref{figgluingdiag}, one can see two distinct edges, one red edge, one blue edge, and two ideal tetrahedra.  Let us start with the red edge.  For the first ideal tetrahedra with shape $z_{1}$, there is no need to apply an isometry, as it is clear its red edge is already aligned with $[0,\infty]$, and its other vertices are $1$ and $z_{1}$ where $z_{1}$ has positive imaginary real part.  The contribution to the first gluing equation from this ideal tetrahedra will be 
\begin{equation}\label{cont1}
2\Log(z_{1}) + 1\Log\left(\frac{1}{1-z_{1}}\right)
\end{equation}
For the second ideal tetrahedra, apply the isometry taking $[0,z_{1},\infty,z_{1}(1-1/z_{2})]$ to $[0,1,\infty,(1-1/z_{2})]$.  This isometry is just multiplication by $1/z_{1}$ and fixes the geodesic $(0,\infty)$.
\begin{figure}
	\begin{center}
	\includegraphics[scale=0.5]{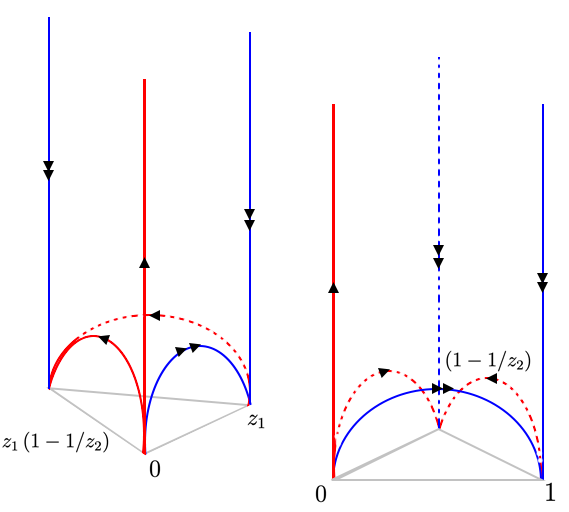}
	\end{center}
	\caption{Rotating the second ideal tetrahedron about the geodesic $[0,\infty]$ to put it in standard position for the gluing equations.  The isometry is $\sigma(w) = w/z_{1}$ which takes $z_{1}(1-1/z_{2})$ to $1-1/z_{2}$.}\label{figT2G1}
\end{figure}
As one can see in Figure \ref{figT2G1}, the contribution to the first gluing equation from this ideal tetrahedron is 
\begin{equation}\label{cont2}
2\Log\left(1 - \frac{1}{z_{2}}\right) + \Log\left(\frac{1}{1 - \left(1 - \frac{1}{z_{2}}\right)}\right) = \Log(z_{2}) + 2 \Log\left(1 - \frac{1}{z_{2}}\right) 
\end{equation}
Summing the contributions from Equation \ref{cont1} and Equation \ref{cont2} yield the desired gluing equation along the red-edge as in Equation \ref{snapglue1}.  Equation \ref{snapglue2} can be obtained in a similar manner by summing contributions from the two tetrahedra along the blue-edge.  Below in Figures \ref{figT2G2} and \ref{T1_G1andG2} are two diagrams that put the tetrahedra into standard position for ease of calculation.  
\begin{figure}
	\begin{center}
	\includegraphics[scale=0.5]{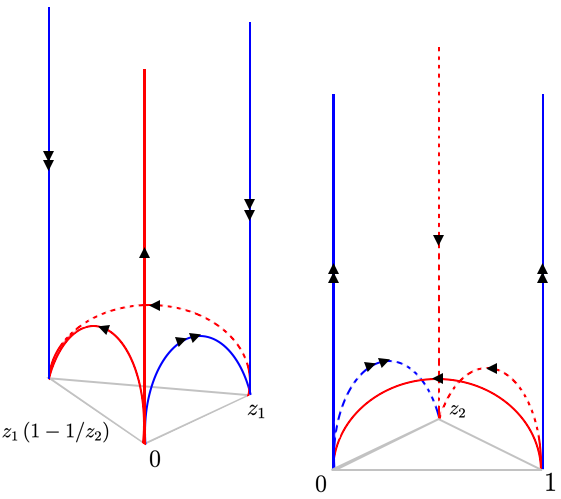}
	\end{center}
	\caption{Isometry applied to the second ideal tetrahedron to put it in standard position for the gluing equations.  The isometry is $\sigma(w) = z_{1}/(z_{1}-w)$ which takes $z_{1}(1-1/z_{2})$ to $z_{2}$.}\label{figT2G2}
\end{figure}
\begin{figure}
	\begin{center}
	\includegraphics[scale=0.5]{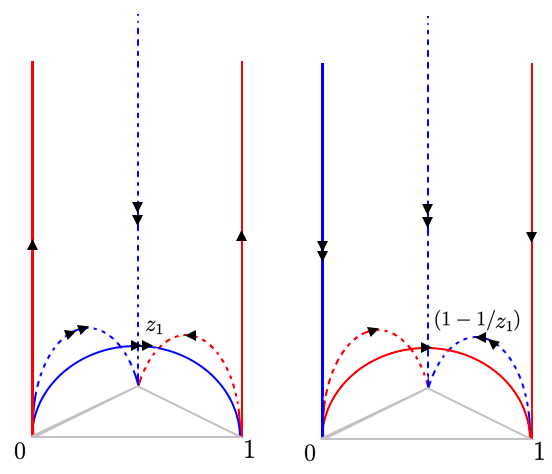}
	\end{center}
	\caption{Isometry applied to the first ideal tetrahedron to put it in standard position for the gluing equations.  The isometry is $\sigma(w) = 1-1/w$ which takes $z_{1}$ to $1-1/z_{1}$.}\label{T1_G1andG2}
\end{figure}
\\
\\
Provided $z_{1}$ and $z_{2}$ are chosen to satisfy Equations \ref{snapglue1} and \ref{snapglue2} and have positive imaginary parts, the gluing pattern as in Figure \ref{figgluingdiag} will produce a hyperbolic structure on the figure-eight knot complement.  This produces a one-parameter family of hyperbolic structures on the figure-eight knot complement parametrized by $z_{1}$, as $z_{2}$ is implicitly defined by $z_{1}$ via the gluing equations in Equations \ref{snapglue1} and \ref{snapglue2}.  Moreover, for each choice of such $z_{1}$ and $z_{2}$, we obtain a representation of the figure-eight knot complement's fundamental group into $\PSL(2,\C)$ via the gluing map isometries.  The case where $ \omega:= e^{\pi i/3} = z_{1} = z_{2}$ corresponds to the unique complete structure on $M$; these are the shapes that satisfy both the gluing and completeness equations of Equation \ref{SNAPPYgluing}.  \\
\\
The gluing maps are generated by isometries $x$ and $y$ which are the maps gluing $D$ to $D'$ and $B$ to $B'$ respectively in Figure \ref{figgluingdiag}.  They are uniquely specified by their actions on the following points in $\del H^{3}$ corresponding to the faces of the two ideal tetrahedra.  
\begin{center}
\begin{minipage}{0.4\textwidth}
\begin{align*}\label{eqgluingmaps}
x(\infty) &= \infty \\
x(z_{1}) &= z_{1}\left(1-\frac{1}{z_{2}}\right)  \\
x(1) &= 0
\end{align*}
\end{minipage}
\begin{minipage}{0.4\textwidth}
\begin{align*}
y(\infty) &= z_{1}\left(1-\frac{1}{z_{2}}\right) \\
y(0) &= 0 \\
 y(1) &= z_{1}
\end{align*}
\end{minipage}
\end{center}
The gluing map taking $C$ to $C'$ is specified uniquely by taking $[0,z_{1},1]$ to $[\infty,z_{1}(1-1/z_{2}),z_{1}]$ which can be readily verified to be given by $yx^{-1}y^{-1}x$, thus $x$ and $y$ generated the fundamental group $\pi_{1}(M)$.  Denoting $a = xy^{-1}x^{-1}y$, a presentation of $\pi_{1}(M)$ is given by
\begin{equation}\label{wirt}
\pi_{1}(M) = \langle x, y \, | \, ax = ya \rangle \text{ where } a = xy^{-1}x^{-1}y
\end{equation}
This yields a representation of the figure-eight knot complement's fundamental group into $\PSL(2,\C)$; namely the holonomy representation.  The corresponding matrix representation of gluing isometries $x$ and $y$ in $\PSL(2,\C)$ are given by
\begin{equation}\label{slmats}
\rho_{u}(x) =\left(
\begin{array}{cc}
\frac{1}{\sqrt{z_1(1-z_2)}} & -\frac{1}{\sqrt{z_1(1-z_2)}}  \\
0 & \sqrt{z_1(1-z_2)}
\end{array}
\right)  \phantom{===}
\rho_{u}(y) = \left(
\begin{array}{cc}
\sqrt{z_1(1-z_2)} & 0 \\
-\frac{z_2}{\sqrt{z_1(1-z_2)}} & \frac{1}{\sqrt{z_1(1-z_2)}}
\end{array}
\right) 
\end{equation}
where $u$ is the complex length of $\rho_{u}(x)$ as defined in Section \ref{ssdehnfill}.  \\
\\
Moreover, if we denote $b = xy^{-1}$ and $l = b^{-1}a^{-1}ba$, then $x$ and $l$ generate the fundamental group of the boundary torus $\del M \subset M$.  The longitude $l$ under our representation is given below.
\begin{equation*}\label{sll}
\rho_{u}(l) = \left(
\begin{array}{cc}
 \frac{z_{2}}{(z_{2}-1)^2} & -\frac{((z_{2}-3) z_{2}+1) (z_{1}
   (z_{2}-1)-z_{2})}{z_{1} (z_{2}-1)^2 z_{2}} \\
 0 & \frac{(z_{2}-1)^2}{z_{2}} \\
\end{array}
\right)
\end{equation*}
 \\
\\
With our choice of generators $x$ and $l$ for $\pi_{1}(\del M)$, following Thurston we let 
\begin{equation}\label{complexlen}
u = \Log\left(z_{1}(1-z_{2})\right) \text{ and } v = 2\Log\left(-\frac{(1-z_{2})^{2}}{z_{2}}\right)
\end{equation}
where the logarithms are branched so that if $z_{1} = z_{2} = \omega$, then $u = v = 0$.  By definition, $u$ and $v$ are the complex lengths of $\rho_{u}(x)$ and $\rho_{u}(l)$.  For $z_{1} \neq \omega$, there exists a unique solution in the real variables $(p,q)$ for which $pu + qv = 2\pi i$.  These are the generalized Dehn filling coefficients of the representation $\rho_{u}$ as defined in Section \ref{ssdehnfill} which parametrize our family of representations into $\PSL(2,\C)$.  With the parametrization which deforms the complete hyperbolic structure of $M$, we move onto parameterizing representations in $\PSO(3,1)$.  
\end{subsection}

\begin{subsection}{Klein model representation}\label{kleinmodelrep}
Let $M$ denote the figure-eight knot complement.  The goal of this section is to construct an explicit family of representations $p_{u}: \pi_{1}(M) \lra \PSO(3,1)$ deforming the discrete and faithful representation of $\rho_{0}\de\pi_{1}(M) \lra \PSO(3,1)$ in such a manner that the representations are rational in the real and imaginary parts of the shapes of the ideal tetrahedra used to glue $M$ together.  If one takes an isomorphism between $\PSL(2,\C)$ and $\PSO(3,1)$ as in \cite[Section 2.1]{Cooper2006Computing} for example, then simply composing the representations yields a deformation; however it is not entirely obvious that this can be done with rational matrix entries in the shape parameters.\\
\\
Let us consider the representation $\rho_{u}(x) \in \PSL(2,\C)$ as defined in Equation \ref{slmats}.  Abusing notation, we will denote this isometry by $x$.  We look at points along the boundary of $\del H^{3}$ that are rational expressions in $z_{1}$ and $z_{2}$ whose image under $x$ is also a rational expression of $z_{1}$ and $z_{2}$.  We call such points in $\del H^{3}$ \emph{rational points for} $x$.  We define rational points for $x$ in $\del K^{3}$ similarly.  Note that a point is rational for $x$ in $\del H^{3}$ if and only if its corresponding point under the isometry $g$ in Equation \ref{isom2} is a rational point for $x$ in $\del K^{3}$.   \\
\\
In theory, three non-collinear points are sufficient to determine an isometry of $K^{3}$, and thus its matrix representation in $\PSO(3,1)$.  Practically however, it is easier to pick five rational points in $\del H^{3}$ for $x$ whose images under $g$ form a projective basis for $\RP^{3}$ and determine the matrix representation this way. \\
\\
Recall, a \emph{projective basis} $\beta \subset \RP^{n}$ is a collection of $(n+2)$-points such that any sub-collection of $(n+1)$-points in $\beta$ determine a linearly independent set in $\R^{n+1}$.  Similar to how a linear isomorphism from $\R^{n+1}$ to itself is determined by its action on a basis of $(n+1)$-points, a projective isomorphism of $\RP^{n}$ is determined by its action on a projective basis of $(n+2)$-points.  \\
\\
Note for each point $p \in \del H^{3}\setminus \infty = \C \subset \hat{\C}$, if we express $p = (x,y,0)$ in its real and imaginary parts, $g$ will take $p$ to 
\begin{equation*}
g(p) = g(x,y,0) = \left(\frac{4x}{4+x^{2}+y^{2}},\frac{4y}{4+x^{2}+y^{2}},\frac{4-x^{2}-y^{2}}{4+x^{2}+y^{2}}\right)
\end{equation*}
whereas $g(\infty)$ gets mapped to $(0,0,-1)$.  We denote $P \in \del K^{3} = S^{2}$ by the image of a point $p \in \del H^{3}$ under $g$; that is $g(p) = P$.  Similarly, if we let $q = x(p)$, we know that because $x$ preserves the boundary $\del H^{3}$, we have that $g(q)$ will also be in $\del K^{3}$.  We denote $Q$ by the image of $q = x(p) \in \del H^{3}$ under $g$.  \\
\\
We choose $5$-rational points $\{p_{1},\hdots, p_{5}\} \subset \del H^{3} = \hat{\C}$ that get mapped to a projective basis $\{P_{1},\hdots, P_{5}\} \subset \del K^{3}$.  We try and choose these $p_{i}$ so that each $p_{i}$ and the corresponding $q_{i}$ have relatively simple rational expressions in $z_{1}$ and $z_{2}$.  This is so that when we write the $p_{i}$'s in terms of their real and imaginary parts as $p_{i} = (x_{i},y_{i},0)$, each $P_{i}$ and $Q_{i}$ is not too unwieldy for Mathematica's symbolic algebra algorithm.    \\
\\
The five points $\{p_{1},p_{2},\hdots, p_{5}\} \subset \del H^{3}$ we choose for $x$ are
\begin{center}\label{xpiqi}
\begin{minipage}{0.4\textwidth}
\begin{align*}
p_{1} &= \infty\\
p_{2} &= 2-z_{2}  \\
p_{3} &= 1 \\
p_{4} &= 1/z_{2} \\
p_{5} &= 1+z_{1}-z_{1}z_{2}
\end{align*}
\end{minipage}
\begin{minipage}{0.4\textwidth}
\begin{align*}
q_{1} &= x(p_{1}) =  \infty\\
q_{2} &= x(p_{2}) =  1/z_{1}  \\
q_{3} &= x(p_{3})  = 0 \\
q_{4} &= x(p_{4})  = 1/(z_{1}z_{2}) \\
q_{5} &= x(p_{5})  = 1
\end{align*}
\end{minipage}
\end{center}
To map the $p_{i}$ and $q_{i}$ through $g$, we expand these points into their real and imaginary parts.  Denote the images of $p_{i}$ under $g$ by $P_{i}$ and the images of $q_{i}$ under $g$ by $Q_{i}$.  We wish to find a matrix in $\PSO(3,1)$ representing the projective map taking $P_{i}$ to $Q_{i}$.  This calculation is long and tedious, and carried out in \texttt{PRFE\textunderscore Mathematica.nb}, however we explain the process here.  Each point $[X:Y:Z:W] \in K^{3}$ has a representative with $W = 1$ corresponding to the affine patch $W \neq 0$.  Thus, we may pick representatives of $P_{i}$ and $Q_{i}$ in $K^{3}$ where the last coordinate is $W = 1$.  Let us denote the representatives for $P_{i}$ by $v_{i} \in \R^{4}$ and the representatives for $Q_{i}$ by $w_{i} \in \R^{4}$.  \\
\\
Because $\{v_{1},\hdots, v_{4}\}$ is a basis for $\R^{4}$, there exists unique scalars $s_{1},\hdots, s_{4} \in \R$ for which 
\begin{equation}\label{v5veceq}
s_{1}v_{1} + s_{2}v_{2} + s_{3}v_{3} + s_{4}v_{4} = v_{5}
\end{equation}
Note that none of the scalars $s_{i} = 0$, for if this were so, then $\{v_{1},v_{2},v_{3},v_{4},v_{5}\}$ would not be a projective basis for $\RP^{3}$.  An analogous statement holds for $\{w_{1},\hdots, w_{4}\}$, $w_{5}$ and the existence of $t_{1},\hdots, t_{4} \in \R$ satisfying
\begin{equation}\label{w5veceq}
t_{1}w_{1} + t_{2}w_{2} + t_{3}w_{3} + t_{4}w_{4} = w_{5}
\end{equation}
We wish to find a linear map $A$ that takes $v_{i}$ to $\lambda_{i} w_{i}$ for some $0 \neq \lambda_{i} \in \R$.  Without loss of generality, we may assume that $Av_{5} = w_{5}$ for if $\lambda_{5} \neq 1$, we may simply replace $A$ by $A/\lambda_{5}$.  Applying $A$ to Equation \ref{v5veceq} yields $Av_{5} = w_{5}$ on the right hand side.  Equating the left hand side of Equation \ref{v5veceq} with $A$ applied to it, and the left hand side of Equation \ref{w5veceq} yields
\begin{align*}
Av_{5} &= A\left(s_{1}v_{1} + s_{2}v_{2} + s_{3}v_{3} + s_{4}v_{4}\right) = s_{1}\lambda_{1}w_{1} + s_{2}\lambda_{2}w_{2} + s_{3}\lambda_{3}w_{3} + s_{4}\lambda_{4}w_{4}  \\
w_{5} &= t_{1}w_{1} + t_{2}w_{2} + t_{3}w_{3} + t_{4}w_{4} 
\end{align*}
Because $\{w_{1},w_{2},w_{3},w_{4}\}$ forms a basis for $\R^{4}$ this uniquely determines each $\lambda_{i} = t_{i}/s_{i}$ for $i = 1\hdots 4$.  From there we may use typical linear algebra to solve for the matrix $A$ taking $v_{i}$ to $\lambda_{i}w_{i}$ for $i = 1\hdots 4$.  \\
\\
This calculation is carried out in \texttt{PRFE\textunderscore Mathematica.nb} and yields the following results, which, we write up to scale.  We adopt the standard notation that $z_{*x}$ is the real part of $z_{*}$ and $z_{*y}$ is the imaginary part of $z_{*}$.  
\begin{tiny}
\begin{equation}\label{projx}
\left(
\begin{array}{cccc}
  8 (z_{1x}-z_{1x} z_{2x}+z_{1y} z_{2y}) 		& -8 (-z_{1y}+z_{1y}z_{2x}+z_{1x}z_{2y})  		& 4 (-z_{1x}+z_{1x}z_{2x}-z_{1y} z_{2y}) 	& 		4 (-z_{1x}+z_{1x}z_{2x}-z_{1y} z_{2y}) \\
 8 (-z_{1y}+z_{1y}z_{2x}+z_{1x}z_{2y}) 		& 8 (z_{1x}-z_{1x} z_{2x}+z_{1y} z_{2y}) 		& 4 (z_{1y}-z_{1y}z_{2x}-z_{1x}z_{2y}) 	& 4 (z_{1y}-z_{1y}z_{2x}-z_{1x}z_{2y})  \\
  4 									& 0 									& 3+4|z_{1}(1-z_{2})|^{2} 				& -5+4|z_{1}(1-z_{2})|^{2} \\
  -4 									& 0 									& -3+4|z_{1}(1-z_{2})|^{2} 				& 5+4|z_{1}(1-z_{2})|^{2} 
\end{array}
\right)
\end{equation}
\end{tiny}
The same method may be applied to find the representation of $y$ in $\PSO(3,1)$ with a different choice of a projective basis.  This is detailed thoroughly in  \texttt{PRFE\textunderscore Mathematica.nb} and the result yields
\begin{tiny}
\begin{equation}\label{projy}
\left(
\begin{array}{cccc}
2 (z_{1x}-z_{1x}z_{2x}+z_{1y}z_{2y})		&-2(z_{1y}-z_{1x}z_{2y}-z_{1y}z_{2x})		&4 \left(z_{1x} z_{2x}+z_{1y}z_{2y}-z_{1x}|z_{2}|^{2}\right)		&-4 \left(z_{1x} z_{2x}+z_{1y}z_{2y}-z_{1x}|z_{2}|^{2}\right)	 \\
2(z_{1y}-z_{1x}z_{2y}-z_{1y}z_{2x})		&2 (z_{1x}-z_{1x}z_{2x}+z_{1y}z_{2y})		&4\left(- z_{1x}z_{2y} + z_{1y}z_{2x}-z_{1y}|z_{2}|^{2} \right)		&-4\left(- z_{1x}z_{2y} + z_{1y}z_{2x}-z_{1y}|z_{2}|^{2} \right)\\
-4z_{2x}								&4z_{2y}								&1-4|z_{2}|^{2} + |z_{1}(1-z_{2})|^{2}							&1 + 4|z_{2}|^{2} - |z_{1}(1-z_{2})|^{2}	\\
-4z_{2x}								&4z_{2y}								&1-4|z_{2}|^{2} - |z_{1}(1-z_{2})|^{2}							&1 + 4|z_{2}|^{2} + |z_{1}(1-z_{2})|^{2}	
\end{array}
\right)
\end{equation}
\end{tiny}
In the parabolic case where $z_1 = z_2 = \omega = (1+i\sqrt{3})/2$, Equation \ref{projx} and \ref{projy} yield
\begin{equation}\label{paraspx}
x = \left(
\begin{array}{cccc}
 8 & 0 & -4 & -4 \\
 0 & 8 & 0 & 0 \\
 4 & 0 & 7 & -1 \\
 -4 & 0 & 1 & 9 \\
\end{array}
\right)\phantom{===} y = \left(
\begin{array}{cccc}
 2 & 0 & 2 & -2 \\
 0 & 2 & -2 \sqrt{3} & 2 \sqrt{3} \\
 -2 & 2 \sqrt{3} & -2 & 4 \\
 -2 & 2 \sqrt{3} & -4 & 6 \\
\end{array}
\right)
\end{equation}
These representations are technically in $\PSO(3,1)$ however, one may take lifts to $\text{SO}(3,1)$ by normalizing the matrices.  Note that in Equation \ref{projx} and \ref{projy}, we have expressed the generators of the fundamental group of $\pi_{1}(M)$ in terms of matrices in $\PSO(3,1)$ whose entries are rational in the real and imaginary parts of the ideal teterahedra that glue together to yield the figure-eight knot complement.  These representations also parametrize the deformation away from the discrete and faithful representation of $\pi_{1}(M)$ into $\PSO(3,1)$.  
\end{subsection}

\begin{subsection}{Cohomology and Projective Deformations}\label{visualballas}
Let $M$ denote the figure-eight knot complement and consider a representation $\rho_{u}: \pi_{1}(M) \lra \PSO(3,1) \subset \PGL(4,\R)$ as in Equations \ref{projx} and \ref{projy} where $u \in U$ is in the hyperbolic Dehn filling space of the figure-eight knot complement.  Following Johnson and Millson, split $\slf(4,\R)$ as $\pi_{1}(M)$-modules \cite[Section 1]{Johnson1987Deformation}.  Take the orthogonal complement of $\so(3,1)$ inside of $\slf(4,\R)$ under the Killing-form to obtain a $\pi_{1}(M)$-complement which we denote $\vv$.  This provides a decomposition of $\slf(4,\R)$ into $\pi_{1}(M)$-modules
\begin{equation*}\label{liesplit}
\slf(4,\R) = \so(3,1) \oplus \vv
\end{equation*}
where $\vv$ is a $\dim \slf(4,\R) - \dim \so(3,1) = 15 - 6 = 9$-dimensional complement.  \\
\\
We wish to understand the deformations of $\rho_{u} \de \pi_{1}(M) \lra \PGL(4,\R)$, thus we are naturally led to study the cohomology group $H^{1}(\pi_{1}(M); \slf(4,\R)_{\Ad \rho_{u}})$ associated to the $\pi_{1}(M)$-module, $\slf(4,\R)_{\Ad \rho_{u}}$ given by the composition of $\rho_{u}$ with the adjoint representation of $\PGL(4,\R)$ as in Theorem \ref{thmweilrigidity}.  Due to the contractibility of the universal cover of $M$, we have an isomorphism between $H^{1}(\pi_{1}(M); \slf(4,\R)_{\Ad \rho})$ and $H^{1}(M; \slf(4,\R)_{\Ad \rho_{u}})$, where the latter is the topological group cohomology as defined in Section \ref{subseclocalcohom}.  We will therefore move between topological and group cohomology freely.  \\
\\
Because the Killing-form restricted to $\vv$ is non-degenerate and $\pi_{1}(M$)-invariant, we obtain a splitting of the cohomology groups as below
\begin{equation*}\label{cohomsplit}
H^{*}(M; \slf(4,\R)_{\Ad \rho_{u}}) = H^{*}(M; \so(3,1)_{\Ad \rho_{u}})\oplus H^{*}(M; \vv_{\Ad \rho_{u}})
\end{equation*}
The factor $H^{*}(M; \so(3,1)_{\Ad \rho_{u}})$ is a well studied object and many properties regarding these cohomology groups may be seen as the algebraic analogues to Thurston's deformation theory of cusped hyperbolic 3-manifolds, \cite[Section 8.8]{Kapovich2010Hyperbolic}.  Hyperbolic deformations of the complete structure on $H^{1}(M; \so(3,1)_{\Ad \rho_{u}})$ are known to be parameterized by the a single complex number for each toroidal boundary component as described in Section \ref{ssdehnfill}.\\
\\
What is less well understood is the complementary term $H^{1}(M; \vv_{\Ad \rho_{u}})$.  Non-hyperbolic deformations give rise to non-zero elements in this cohomology group.  Consider the composition $\Ad \rho_{u}: \pi_{1}(M) \lra \PSO(3,1) \lra \Aut(\vv)$.  Having expressed the generators in $\PSO(3,1)$ as matrices with rational entries in the real and imaginary parts of $z_{1}$ and $z_{2}$, we aim to do the same with the images of $x$ and $y$ in $\PSO(3,1)$ in $\Aut(\vv)$ under the adjoint representation.  \\
\\
To this end, we first choose the following basis $\beta = \{v_{1},\hdots, v_{9}\}$ for $\vv$.   
\begin{align}\label{vbasis}
v_{1} &=\left(
\begin{array}{cccc}
 1 & 0 & 0 & 0 \\
 0 & 0 & 0 & 0 \\
 0 & 0 & 0 & 0 \\
 0 & 0 & 0 & -1 \\
\end{array}
\right) \phantom{=}
v_{2} =\left(
\begin{array}{cccc}
 0 & 0 & 0 & 0 \\
 0 & 1 & 0 & 0 \\
 0 & 0 & 0 & 0 \\
 0 & 0 & 0 & -1 \\
\end{array}
\right) \phantom{=}
v_{3} =\left(
\begin{array}{cccc}
 0 & 0 & 0 & 0 \\
 0 & 0 & 0 & 0 \\
 0 & 0 & 1 & 0 \\
 0 & 0 & 0 & -1 \\
\end{array}
\right) \nonumber \\
v_{4} &= \left(
\begin{array}{cccc}
 0 & 1 & 0 & 0 \\
 1 & 0 & 0 & 0 \\
 0 & 0 & 0 & 0 \\
 0 & 0 & 0 & 0 \\
\end{array}
\right)\phantom{=}
v_{5} = \left(
\begin{array}{cccc}
 0 & 0 & 1 & 0 \\
 0 & 0 & 0 & 0 \\
 1 & 0 & 0 & 0 \\
 0 & 0 & 0 & 0 \\
\end{array}
\right) \phantom{=}
v_{6} =\left(
\begin{array}{cccc}
 0 & 0 & 0 & -1 \\
 0 & 0 & 0 & 0 \\
 0 & 0 & 0 & 0 \\
 1 & 0 & 0 & 0 \\
\end{array}
\right) \nonumber \\
v_{7} &= \left(
\begin{array}{cccc}
 0 & 0 & 0 & 0 \\
 0 & 0 & 1 & 0 \\
 0 & 1 & 0 & 0 \\
 0 & 0 & 0 & 0 \\
\end{array}
\right)\phantom{=}
v_{8} =\left(
\begin{array}{cccc}
 0 & 0 & 0 & 0 \\
 0 & 0 & 0 & -1 \\
 0 & 0 & 0 & 0 \\
 0 & 1 & 0 & 0 \\
\end{array}
\right)\phantom{=}
v_{9} =\left(
\begin{array}{cccc}
 0 & 0 & 0 & 0 \\
 0 & 0 & 0 & 0 \\
 0 & 0 & 0 & -1 \\
 0 & 0 & 1 & 0 \\
\end{array}
\right)
\end{align}
The adjoint action of $g \in \PSO(3,1)$ on $\vv$ is simply conjugation.  That is, $\Ad g$ acts on the vector $v \in \vv$ via $gvg^{-1}$.  Pre-composing this action with the representation $\rho_{u}: \pi_{1}(M) \lra \PSO(3,1)$ yields a $\pi_{1}(M)$-action on $\vv$ via 
\begin{equation}\label{eqadact}
\left[\Ad \rho_{u}(\gamma)\right](v) := \rho_{u}(\gamma)v\rho_{u}(\gamma)^{-1} \text{ for each } \gamma \in \pi_{1}(M)
\end{equation}
Recall that $\pi_{1}(M)$ is generated by $x$ and $y$, and both $\rho_{u}(x)$ and $\rho_{u}(y)$ are rational in in the real and imaginary parts of the shapes of the ideal tetrahedra, $z_{1}$ and $z_{2}$, used to glue $M$ together as in Figure \ref{figgluingdiag}.  The matrix representations of $\Ad \rho_{u}(x)$ and $\Ad \rho_{u}(y)$ relative to $\beta$ are thus rational in the real and imaginary parts of $z_{1}$ and $z_{2}$, as both $\rho_{u}(x)$ and $\rho_{u}(y)$ are, and because the adjoint representation is given by matrix conjugation.  \\
\\
Thus we able to express the representation $\Ad \rho_{u}: \pi_{1}(M) \lra \Aut(\vv)$ as a rational representation in the real and imaginary parts of $z_{1}$ and $z_{2}$.  The explicit matrices can be found in \texttt{PRFE\underscore Mathematica\underscore File.nb}.  They are the variables \texttt{adpX} and \texttt{adpY} respectively.  The rational representation property will become particularly useful in later sections as we formally verify the rigidity of many Dehn-surgeries on the figure-eight knot complement.
\end{subsection}

\begin{subsection}{Ballas Deformation}\label{ssballas}
In this subsection we provide explicit deformations of the complete hyperbolic structure of the figure-eight knot complement constructed by Ballas.  We illustrate the deformation through totally geodesic slices of the convex domains.  Through his construction of normal forms of two generator free groups with irreducible parabolic image in $\PSO(3,1) \subset \PGL(4,\R)$, Ballas provided the family of deformations of the figure-eight knot complement given below  \cite[Section 5.1]{Ballas2014Deformations}.  
\begin{equation}\label{BallasDeform}
\rho(x)_{t} := \left(
\begin{array}{cccc}
1 & 0 & 0 & 0\\
t & 1 & 0 & 0 \\
2 & 1 & 1 & 0 \\
1 & 1 & 0 & 1
\end{array}
\right)
\text{ and } 
\rho(y)_{t} := \left(
\begin{array}{cccc}
1 & 0 & 1 & \frac{3-t}{t-2} \\
0 & 1 & 1 & \frac{1}{t-2} \\
0 & 0 & 1 & \frac{t}{2(t-2)} \\
0 & 0 & 0 & 1
\end{array}
\right) 
\end{equation}
Here $t = 4$ corresponds to the complete hyperbolic representation.  It is worth noting, we switched the generators $x_{t}$ and $y_{t}$ originally provided in \cite{Ballas2014Deformations} in order to agree with our presentation of $\pi_{1}(M)$ given in Equation \ref{wirt}.  To see this is a non-trivial deformation, one may calculate the trace of the longitude.  If we denote $a = xy^{-1}x^{-1}y$ and $b = xy^{-1}$, then a longitude of the boundary torus of $M$ that commutes with $x$ is given by $l = b^{-1}a^{-1}ba = (yx^{-1}y^{-1}x)(xy^{-1}x^{-1}y)$.  This calculation yields that $\tr \rho_{t}(l) = (48+(t-2)^{4})/8(t-2)$.  Thus we see for different values of $t$, we are indeed obtaining non-conjugate representations, otherwise the trace would remain constant.  One can calculate there is only one value of $t$ for which $l$ is parabolic, namely $t = 4$.  This says that there is no such projective deformation if we require the longitude to remain parabolic throughout the deformation.  Ballas's observation is validated by Corollary \ref{deforml} in Section \ref{oncepuncturedtorus} which states this is necessarily the case.  \\
\\
We conclude this subsection by providing some illustrations of Ballas's deformation.  In \texttt{Projective\underscore Tetrahedra} we wrote a program that illustrates the deformation in the universal cover of $M$.  Let $\Gamma_{t}: =\rho_{t} \, \pi_{1}(M) \subset \PGL(4,\R)$.  We tile the $\Gamma_{t}$-invariant convex body $\Omega_{t} \subset \RP^{3}$ by constructing two projective tetrahedra $T_{1t}$ and $T_{2t}$ inside $\RP^{3}$ and depict their orbits under $\Gamma_{t}$.  We first observe that $\Omega_{4}$ is the projectivization of $||x|| < 0$ under the bilinear form induced by the matrix
\begin{equation*}
\left(
\begin{array}{cccc}
 1 & -1 & 1 & -4 \\
 -1 & 1 & -1 & 1 \\
 1 & -1 & 4 & -4 \\
 -4 & 1 & -4 & 4 \\
\end{array}
\right)
\end{equation*}
We define by the following vertices which are the fixed points of $y_{t}$, $y_{t}^{-1}x_{t}y_{t}$, $x_{t}$, $x_{t}y_{t}^{-1}x_{t}y_{t}x_{t}^{-1}$ and $x_{t}y_{t}x_{t}^{-1}$ respectively, where here we are denoting $x_{t}$ and $y_{t}$ abusively by $\rho_{t}(x)$ and $\rho_{t}(y)$ respectively.  
\begin{footnotesize}
\begin{equation}
p_{1t} := \left[\begin{array}{c}-1\\3\\0\\0\end{array}\right] \phantom{=}
p_{2t} := \left[\begin{array}{c}-8\\8\\-\frac{8 (t-4)}{t-2}\\-16\end{array}\right] \phantom{=}
p_{3t} := \left[\begin{array}{c}0\\0\\-1\\-1\end{array}\right] \phantom{=}
p_{4t} := \left[\begin{array}{c}t-2\\t^2-3 t+2 \\2 t-6 \\2 t-4\end{array}\right] \phantom{=} 
p_{5t} := \left[\begin{array}{c}1\\1+t\\3\\2\end{array}\right] 
\end{equation}
\end{footnotesize}
Define $T_{1t}:=[p_{1t},p_{2t},p_{3t},p_{4t}]$ and $T_{2t} := [p_{1t},p_{3t},p_{4t},p_{5t}]$, and then inspect their orbits under $\Gamma_{t}$.  In the following pictures, we slice $\Omega_{t}$ by totally geodesic subspaces and depict their intersections with the two tetrahedra.  The orbits of $T_{1t}$ are depicted in red whereas the orbits of $T_{2t}$ are depicted in blue.  \\
\\
In the images produced below, we use totally geodesic slices for different heights relative to the chart for $\RP^{3}$ induced by the hyperplane $-\frac{3}{4} x + \frac{1}{4} y - \frac{1}{2} z + \frac{3}{16} w = 1$.  Specifically, the chart from $\phi: \R^{3} \lra \RP^{3}$ is given by
\begin{equation*}\label{chart}
\phi(y,z,w) := \left[
\begin{array}{c}
 -\frac{4}{3} \left(-\frac{3 w}{16}-\frac{y}{4}+\frac{z}{2}+1\right) \\
 y \\
 z \\
 w \\
\end{array}
\right]
\end{equation*}
In Figure \ref{t4000}, we depict several different heights for values of $w$ slicing through $\Omega_{4}$, a projective equivalent of the Klein model corresponding to the hyperbolic structure on $M$, tiled by our ideal tetrahedra.  
\begin{figure}
	\begin{center}
	\includegraphics[scale=0.225]{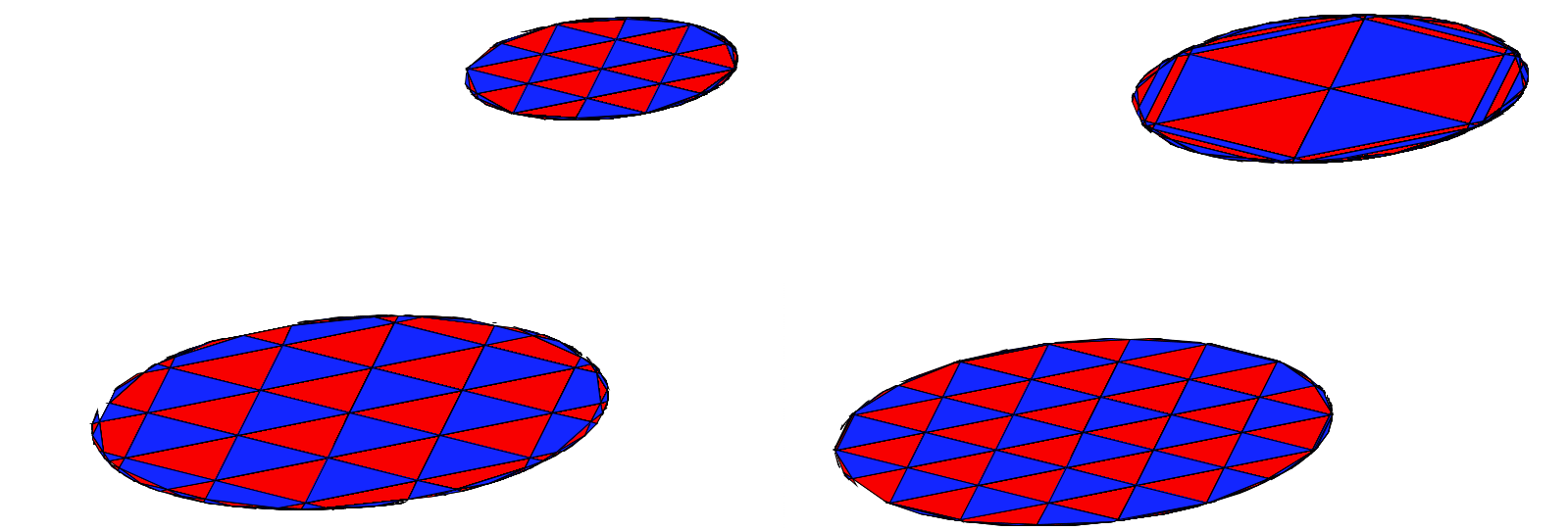}
	\end{center}
	\caption{Totally geodesic slices of $\Omega_{4}$ at different values of $w$.  Reading left to right, top to bottom, the values of $w$ are $-1.2,-2.85,-8.65,$ and $-10.25$.  This corresponds to the hyperbolic structure on $M$.  Note as we slice through close to a cusp, e.g. $w = -10.25$, the parallelism classes of the edges of the ideal tetrahedra are preserved.}\label{t4000}
\end{figure}
Figure \ref{N02200deform} is a fixed geodesic slice, $w = -2.2$, undergoing deformation for different values of $t$.  
\begin{figure}
	\begin{center}
	\includegraphics[scale=0.225]{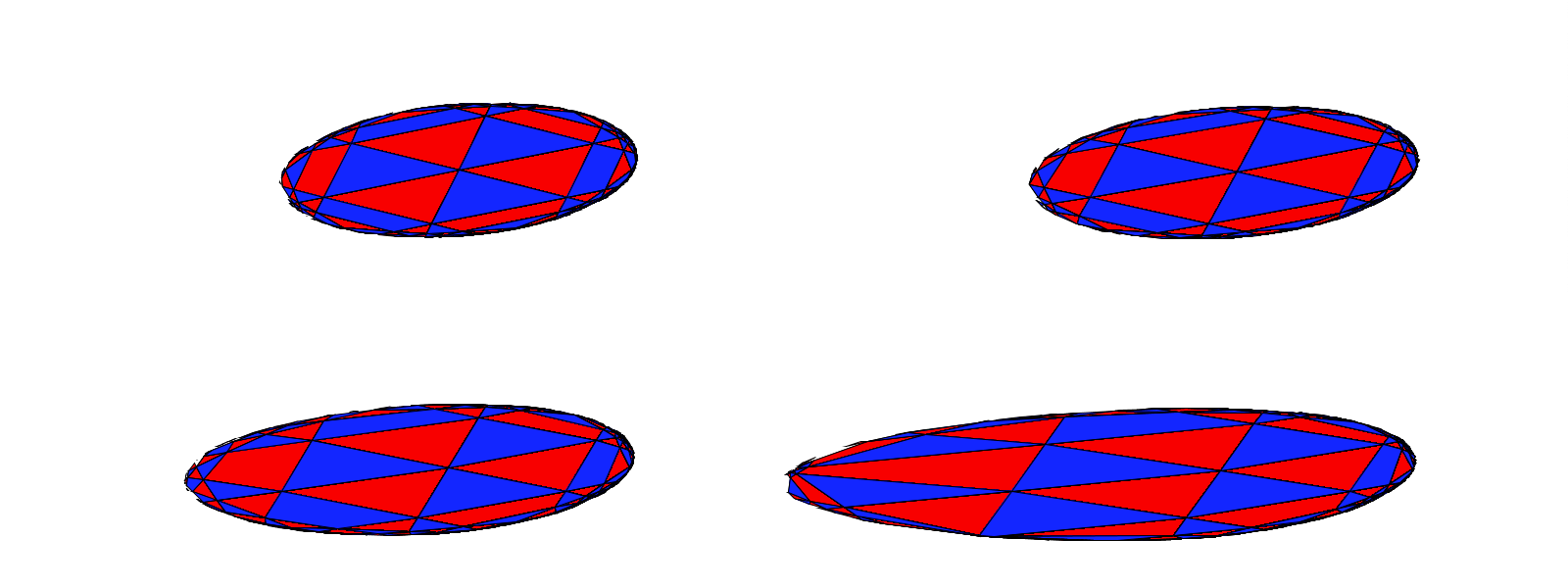}
	\end{center}
	\caption{Totally geodesic slice $w = -2.2$ of $\Omega_{t}$ for different values of $t$.  Reading left to right, top to bottom, the values of $t$ are $4.0, 4.13, 4.325$ and $4.75$.  Note that as $t$ deforms, the parallelism classes of the ideal tetrahedra do as well.}\label{N02200deform}
\end{figure}
Figure \ref{N0060deform} shows a fixed geodesic slice, $w = -0.6$, undergoing deformation.  This is a slice closer to the boundary of the Klein model $\Omega_{4}$. 
\begin{figure}
	\begin{center}
	\includegraphics[scale=0.225]{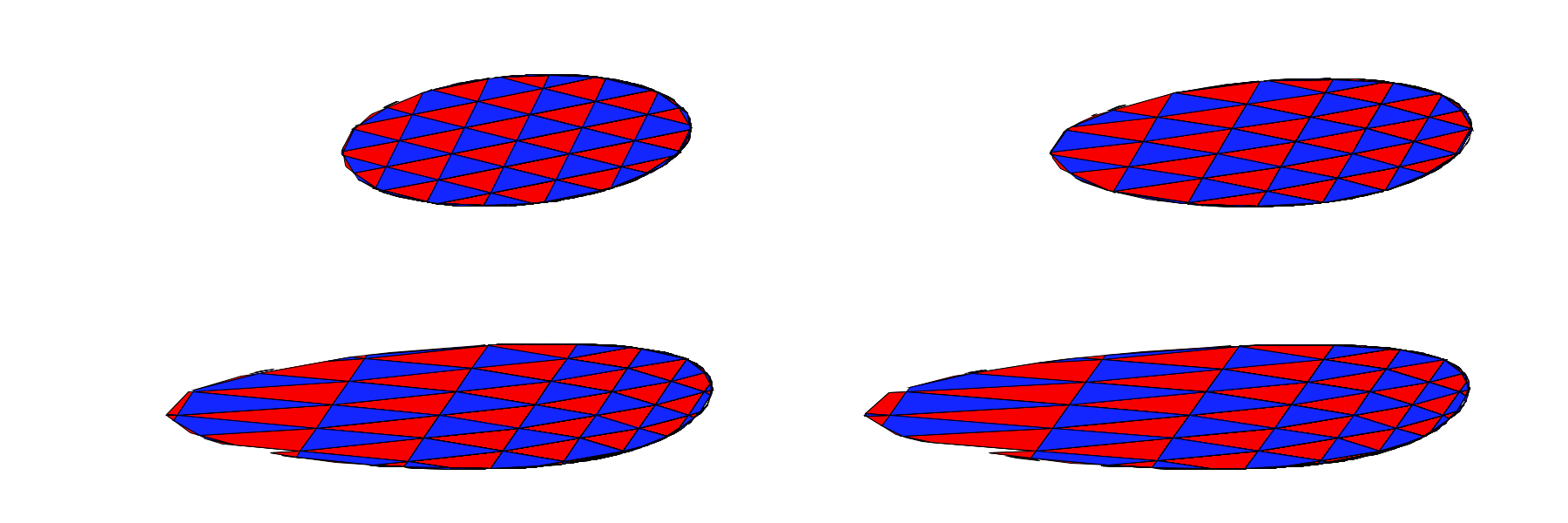}
	\end{center}
	\caption{Totally geodesic slice $w = -0.6$ of $\Omega_{t}$ for different values of $t$.  Reading left to right, top to bottom, the values of $t$ are $4.0, 4.275, 4.595$ and $4.75$.}\label{N0060deform}
\end{figure}
Finally in Figure \ref{t45800} we provide several different heights for values of $w$ slicing through $\Omega_{4.58}$.  Note we as slice through the various $w$, and particularly as we draw closer to a cusp of $\Omega_{4.58}$, one can see the parallelism classes of the edges of the ideal tetrahedra are unpreserved and distorted.  Contrast this behavior with Figure \ref{t4000}.  
\begin{figure}
	\begin{center}
	\includegraphics[scale=0.225]{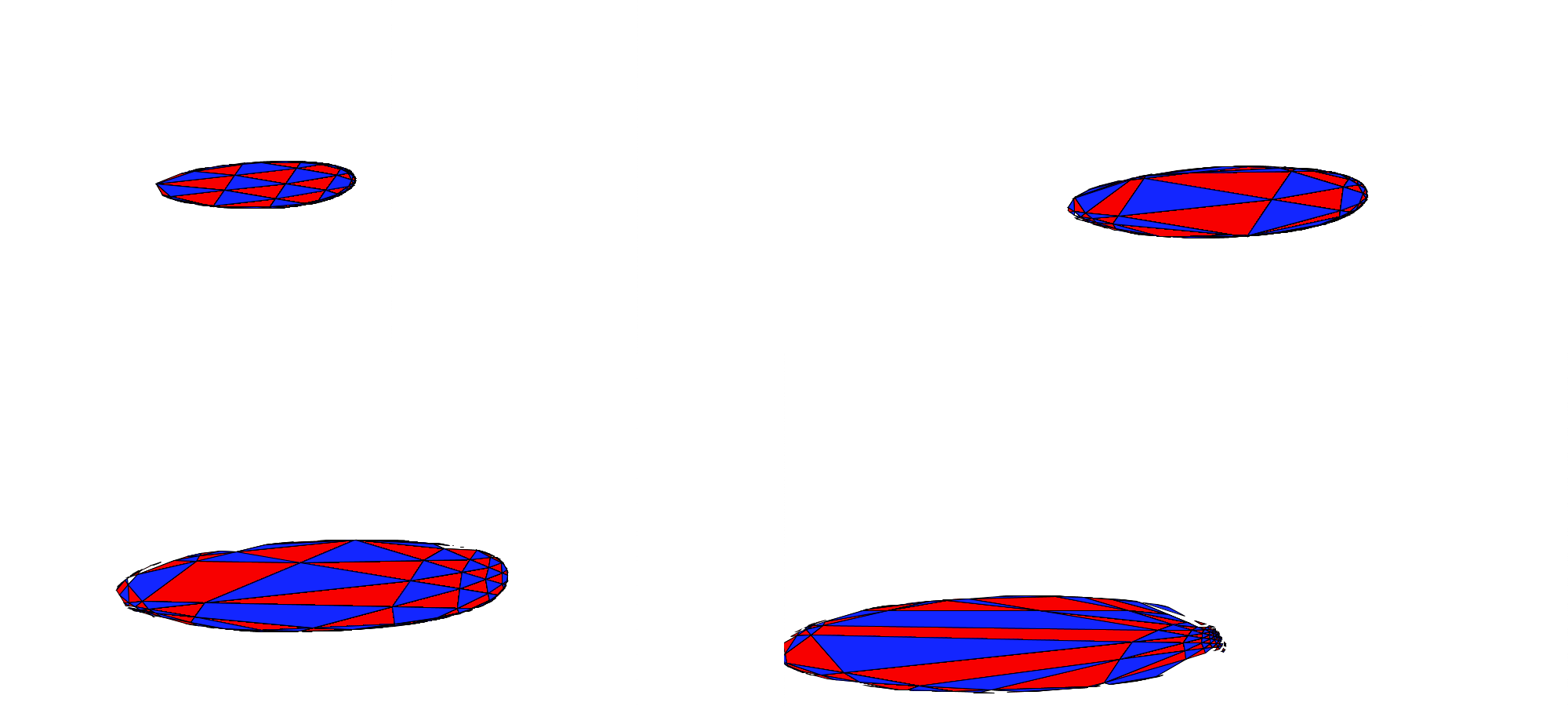}
	\end{center}
	\caption{Totally geodesic slices of $\Omega_{4.58}$ at different values of $w$.  Reading left to right, top to bottom, the values of $w$ are $-1.75,-4.35,-9.05,$ and $-18.20$.}\label{t45800}
\end{figure}
It is worth mentioning we did not generate these ideal tetrahedra in any sophisticated manner.  We simply generated about $35,000$ random words in $\rho_{t}(x)$ and $\rho_{t}(y)$ of length at most $15$ and depicted the results.  Had we generated these pictures using more elaborate search algorithms, it is likely we could avoid the gaps as most appreciably seen in Figure \ref{t45800}, however we simply wanted to provide a program that is quick and illustrates the tiling sufficiently well.  
\end{subsection}
\end{section}

\begin{section}{Cohomology of the Figure-Eight Knot}\label{secchmfig8}
In this section we study the cohomology the figure-eight knot complement $M$ twisted by the representation $\rho_{u} \de \pi_{1}(M) \lra \Aut(\vv)$ by using the fiber-bundle structure $F \lra M \lra S^{1}$ where $F$ is a once-punctured torus.  We prove that non-trivial projective deformations of the figure-eight knot complement induce non-trivial projective deformations of the once-punctured torus fiber.  Moreover, we show that for $u \in U$ sufficiently close to the complete hyperbolic structure on $M$, that non-trivial deformations of the figure-eight knot complement induce non-trivial deformations of the longitude of the boundary torus.  
\begin{subsection}{Cohomology of the circle}
Let $M$ denote the figure-eight knot complement, and let $x$ and $l$ denote a meridian and longitude of the boundary torus $\del M$ respectively.  In terms of our generators in Equation \ref{wirt}, if we define $a = xy^{-1}x^{-1}y$ and $b = xy^{-1}$, we may take $l = b^{-1}a^{-1}ba$.  This expression for $l$ hints at the fact that it is homologically trivial in $H_{1}(M)$.  Observe that $x$ and $l$ generate the fundamental group of the boundary torus $\del M = T^{2}$ and we are interested in calculating the cohomology groups $H^{*}(x;V)$ and $H^{*}(l;V)$.  The following lemma assists our direction.  
\begin{lemma}\label{circcohom}
Let $\rho: \langle x\rangle  \lra \Aut(V)$ be a representation of the infinite cyclic group $\langle x \rangle$ into a finite dimensional vector space $V$.  Then $H^{0}(\langle x \rangle ; V)$ is naturally isomorphic to the fixed points $V^{x}$ whereas $H^{1}(\langle x \rangle; V)$ is naturally isomorphic to the co-invariants $V_{x} := V/\emph{im} (1-x)$.  
\end{lemma}
A proof of this can be found in any introductory text on group cohomology such as \cite[Section III.1 Example 1]{Brown1982Cohomology}.  In a typical abuse of notation we frequently write $H^{*}(x;V)$ instead of $H^{*}(\langle x \rangle ; V)$.  Note Lemma \ref{circcohom} yields that $H^{0}(x; V)$ and $H^{1}(x; V)$ have the same dimension, whereas all other cohomology groups vanish because $H^{*}(x; V) \simeq H^{*}(S^{1}; V)$ and $S^{1}$ is one-dimensional.  \\
\\
In what follows, we pay particular attention to our representation of $x \in \pi_{1}(M)$ into $\PSL(2,\C)$ as in Equation \ref{slmats} 
\begin{equation*}
\rho_{u}(x) = \left(
\begin{array}{cc}
\frac{1}{\sqrt{z_1(1-z_2)}} & -\frac{1}{\sqrt{z_1(1-z_2)}}  \\
0 & \sqrt{z_1(1-z_2)}
\end{array}
\right)  
\end{equation*}
however this discussion applies more broadly to the case where $\rho_{u}(x)$ is a non-trivial loxodromic element of $\SL(2,\C)$.  The eigenvalues of $x$ are $\pm \sqrt{z_{1}(1-z_{2})}$.  We wish to study the eigenvalues of the corresponding automorphism of $\vv$ under the adjoint representation of $\PGL(4,\R)$.  We know that in $\PSO(3,1)$, up to conjugation, the representation of $x$ away from the parabolic point is
\begin{equation}\label{psorep}
\left(
\begin{array}{cccc}
\cos(\theta) & -\sin(\theta) & 0 & 0  \\
\sin(\theta) & \cos(\theta) & 0 & 0  \\
0 & 0 & \cosh(c) & \sinh(c)  \\
0 & 0 & \sinh(c) & \cosh(c)  \\
\end{array}
\right)  
\end{equation}
where 
\begin{equation*}
u = \Log z_{1}(1-z_{2}) = \Log | z_{1}(1-z_{2})| + i\arg  z_{1}(1-z_{2}) = c + i\theta
\end{equation*}
defines $c$ and $\theta$ by the real and imaginary parts of $u$ respectively.  That is, $e^{c}$ is the hyperbolic translations length along the fixed geodesic corresponding the hyperbolic transformation $\rho_{u}(x)$ in Equation \ref{slmats}, and $\theta$ is the angle of rotation about this fixed geodesic.  The matrix representation in Equation \ref{psorep} follows from taking our element $\rho_{u}(x) \in \PSL(2,\C)$ and diagonalizing as in \cite[Section 2.2]{Cooper2006Computing}.  This can be done so long as $z_{1}(1-z_{2}) \neq 1$ which is precisely the condition that the representation is not the complete parabolic one.  The eigenvalues of the matrix in Equation \ref{psorep} are equal to $e^{i\theta}, e^{-i\theta}, e^{c}$, and $e^{-c}$.  \\
\\
We consider the action of the matrix in Equation \ref{psorep} on $\vv$ relative to the basis below. 
\begin{align*}\label{eigenbasisx}
w_{1} &=\left(
\begin{array}{cccc}
 0 & 0 & 0 & 0 \\
 0 & 0 & 0 & 0 \\
 0 & 0 & -1 & -1 \\
 0 & 0 & 1 & 1 \\
\end{array}
\right) \phantom{=}
w_{2} = \left(
\begin{array}{cccc}
 0 & 0 & -1 & -1 \\
 0 & 0 & 0 & 0 \\
 -1 & 0 & 0 & 0 \\
 1 & 0 & 0 & 0 \\
\end{array}
\right) \phantom{=}
w_{3} = \left(
\begin{array}{cccc}
 0 & 0 & 0 & 0 \\
 0 & 0 & -1 & -1 \\
 0 & -1 & 0 & 0 \\
 0 & 1 & 0 & 0 \\
\end{array}
\right) \\
w_{4} &= \left(
\begin{array}{cccc}
 -1 & 0 & 0 & 0 \\
 0 & -1 & 0 & 0 \\
 0 & 0 & 1 & 0 \\
 0 & 0 & 0 & 1 \\
\end{array}
\right) \phantom{=}
w_{5} = \left(
\begin{array}{cccc}
 0 & 1 & 0 & 0 \\
 1 & 0 & 0 & 0 \\
 0 & 0 & 0 & 0 \\
 0 & 0 & 0 & 0 \\
\end{array}
\right) \phantom{=}
w_{6} = \left(
\begin{array}{cccc}
 -1 & 0 & 0 & 0 \\
 0 & 1 & 0 & 0 \\
 0 & 0 & 0 & 0 \\
 0 & 0 & 0 & 0 \\
\end{array}
\right) \\
w_{7} &= \left(
\begin{array}{cccc}
 0 & 0 & 0 & 0 \\
 0 & 0 & 0 & 0 \\
 0 & 0 & 1 & -1 \\
 0 & 0 & 1 & -1 \\
\end{array}
\right) \phantom{=}
w_{8} =\left(
\begin{array}{cccc}
 0 & 0 & 0 & 0 \\
 0 & 0 & 1 & -1 \\
 0 & 1 & 0 & 0 \\
 0 & 1 & 0 & 0 \\
\end{array}
\right) \phantom{=}
w_{9} = \left(
\begin{array}{cccc}
 0 & 0 & -1 & 1 \\
 0 & 0 & 0 & 0 \\
 -1 & 0 & 0 & 0 \\
 -1 & 0 & 0 & 0 \\
\end{array}
\right)
\end{align*}
Relative to this basis the action of $\Ad \rho_{u}(x)$ on $\vv$ is given by
\begin{equation}\label{xactioneig}
\left(
\begin{array}{ccccccccc}
 e^{-2 c} & 0 & 0 & 0 & 0 & 0 & 0 & 0 & 0 \\
 0 & e^{-c} \cos (\theta) & -e^{-c} \sin (\theta)  & 0 & 0 & 0 & 0 & 0 & 0 \\
 0 & e^{-c} \sin (\theta) & e^{-c} \cos (\theta) & 0 & 0 & 0 & 0 & 0 & 0 \\
 0 & 0 & 0 & 1 & 0 & 0 & 0 & 0 & 0 \\
 0 & 0 & 0 & 0 & \cos (2 \theta) & -\sin (2 \theta) & 0 & 0 & 0 \\
 0 & 0 & 0 & 0 & \sin (2 \theta) & \cos (2 \theta) & 0 & 0 & 0 \\
 0 & 0 & 0 & 0 & 0 & 0 & e^{2 c} & 0 & 0 \\
 0 & 0 & 0 & 0 & 0 & 0 & 0 & e^c \cos (\theta) & -e^c \sin (\theta) \\
 0 & 0 & 0 & 0 & 0 & 0 & 0 & e^c \sin (\theta) & e^c \cos (\theta) \\
\end{array}
\right)
\end{equation}
This above representation yields the eigenvalues of $\rho_{u}(x)$; they are given by 
\begin{equation}\label{eigofadx}
e^{-2c},e^{-c+i\theta}, e^{-c-i\theta}, 1, e^{2i\theta}, e^{-2i\theta}, e^{2c}, e^{c+i\theta}, e^{c-i\theta}
\end{equation}
Note that the above calculates the eigenvalues for any non-trivial loxodromic element of $\PSO(3,1)$ where $e^{c}$ and $e^{-c}$ are the hyperbolic translations length along the fixed geodesic and $\theta$ is the argument of the rotation along the fixed geodesic.  \\
\\
In the case where $u = 0$, $\rho_{0}(x) \in \PSL(2,\C)$ is parabolic.  The Jordan normal form of $\rho_{0}(x)$ in $\PSO(3,1)$ is
\begin{equation*}
\left(
\begin{array}{cccc}
 1 & 1 & 0 & 0 \\
 0 & 1 & 1 & 0 \\
 0 & 0 & 1 & 0 \\
 0 & 0 & 0 & 1 \\
\end{array}
\right)
\end{equation*}
A straightforward, albeit, tedious calculation shows that the Jordan normal form of $\Ad \rho_{0}(x)$ acting on $\vv$ is
\begin{equation}\label{PX}
\left(
\begin{array}{ccccccccc}
 1 & 0 & 0 & 0 & 0 & 0 & 0 & 0 & 0 \\
 0 & 1 & 1 & 0 & 0 & 0 & 0 & 0 & 0 \\
 0 & 0 & 1 & 1 & 0 & 0 & 0 & 0 & 0 \\
 0 & 0 & 0 & 1 & 0 & 0 & 0 & 0 & 0 \\
 0 & 0 & 0 & 0 & 1 & 1 & 0 & 0 & 0 \\
 0 & 0 & 0 & 0 & 0 & 1 & 1 & 0 & 0 \\
 0 & 0 & 0 & 0 & 0 & 0 & 1 & 1 & 0 \\
 0 & 0 & 0 & 0 & 0 & 0 & 0 & 1 & 1 \\
 0 & 0 & 0 & 0 & 0 & 0 & 0 & 0 & 1 \\
\end{array}
\right)
\end{equation}
These calculations are carried out in \texttt{PRFE\textunderscore Mathematica.nb} and in particular can be found in the variables \texttt{PspX} and \texttt{PX}.  Using the forms of Equation \ref{xactioneig} and \ref{PX} acting on $\vv$, we expand upon a specific case as proven in \cite[Lemma 4.2 \& 4.4]{Heusener2011Infinitesimal}.  
\begin{lemma}\label{circdim}
Let $M$ be the figure-eight knot complement and $\rho_{u}\de\pi_{1}(M) \lra \PSO(3,1)$ be a representation where $u \in U$ is an element of the hyperbolic Dehn filling space.  For each representation $\Ad \rho_{u}\de\pi_{1}(M) \lra \Aut(\vv)$ the dimension of $H^{*}(x;\vv_{\Ad \rho_{u}})$ is either 1, 3, or 5.  In hyperbolic Dehn filling space, the subset all of representations for which the dimension is equal to 1 is open, whereas the subset of all representations for which the dimension is equal to 5 are discrete isolated points.  The subsets for which the dimension is equal to 3 is neither open nor closed, but are locally topologically equivalent to a curve, or, the crossing of two curves.  The representation $\rho_{0}$ for example lies in the intersection of two curves.  
\end{lemma}
\begin{proof}
To ease notation, we denote the $\pi_{1}(M)$-module $\vv$ associated to $\Ad \rho_{u}\de\pi_{1}(M) \lra \Aut(\vv)$ by $\vv_{u}$.  By Lemma \ref{circcohom}, the dimension of $H^{*}(x;\vv_{u})$ is equal to the dimension of the fixed points of the action $\Ad \rho_{u}(x) \de \vv \lra \vv$ because $x$ is infinite cyclic in $\pi_{1}(M)$.  Let $\rho_{u}\de\pi_{1}(M) \lra \Aut(\vv)$ be a representation where $u \in U$ is a point in hyperbolic Dehn filling space $U$.  By definition $u = \Log z_{1}(1-z_{2}) = c + i\theta$.  We proceed by cases.  For now, let us assume that $u \neq 0$, so either $c \neq 0$ or $\theta \neq 0$.  If $c$ is non-zero, and $\theta$ is not an integer multiple of $\pi$, then all the eigenvalues as seen in Equation \ref{eigofadx} are distinct because $e^{-c} < 1 < e^{c}$ and $e^{i\theta} \neq e^{-i\theta}$.  This condition is an open condition, as if we pick a $u \in U$ satisfying $c$ is non-zero and $\theta$ is not an integer multiple of $\pi$, we may find an entire neighborhood about $u$ satisfying this condition too.\\
\\
If $c = 0$ and $\theta$ is not an integer multiple of $\pi$, then we still have that $e^{i\theta} \neq e^{-i\theta}$.  In this case the matrix representation for $\Ad \rho_{u}(x)$ becomes
\begin{equation}\label{xhyplengthzero}
\left(
\begin{array}{ccccccccc}
 1 & 0 & 0 & 0 & 0 & 0 & 0 & 0 & 0 \\
 0 & \cos (\theta ) & -\sin (\theta ) & 0 & 0 & 0 & 0 & 0 & 0 \\
 0 & \sin (\theta ) & \cos (\theta ) & 0 & 0 & 0 & 0 & 0 & 0 \\
 0 & 0 & 0 & 1 & 0 & 0 & 0 & 0 & 0 \\
 0 & 0 & 0 & 0 & \cos (2 \theta ) & -\sin (2\theta ) & 0 & 0 & 0 \\
 0 & 0 & 0 & 0 & \sin (2 \theta ) & \cos (2 \theta ) & 0 & 0 & 0 \\
 0 & 0 & 0 & 0 & 0 & 0 & 1 & 0 & 0 \\
 0 & 0 & 0 & 0 & 0 & 0 & 0 & \cos (\theta ) & -\sin (\theta ) \\
 0 & 0 & 0 & 0 & 0 & 0 & 0 & \sin (\theta ) & \cos (\theta ) \\
\end{array}
\right)
\end{equation}
and we can clearly see that $x$ has precisely three fixed points.  In this case, $\dim H^{*}(x; \vv_{u}) = 3$.  Moreover, we have one-dimension of freedom to move $\theta$ and still avoid integer multiples of $\pi$ while keeping $c = 0$.\\
\\
If $c \neq 0$ and $\theta = \pi k$ for some $k \in \Z$, then we have $e^{-c} < 1 < e^{c}$ and $e^{2i\theta} = 1$.  Our matrix representation for $\Ad \rho_{u}(x)$ becomes 
\begin{equation}\label{xrotpi}
\left(
\begin{array}{ccccccccc}
 e^{-2 c} & 0 & 0 & 0 & 0 & 0 & 0 & 0 & 0 \\
 0 & e^{-c} \cos (\pi  k) & 0  & 0 & 0 & 0 & 0 & 0 & 0 \\
 0 & 0 & e^{-c} \cos (\pi  k) & 0 & 0 & 0 & 0 & 0 & 0 \\
 0 & 0 & 0 & 1 & 0 & 0 & 0 & 0 & 0 \\
 0 & 0 & 0 & 0 & 1 & 0& 0 & 0 & 0 \\
 0 & 0 & 0 & 0 & 0 & 1 & 0 & 0 & 0 \\
 0 & 0 & 0 & 0 & 0 & 0 & e^{2 c} & 0 & 0 \\
 0 & 0 & 0 & 0 & 0 & 0 & 0 & e^c \cos (\pi  k) & 0\\
 0 & 0 & 0 & 0 & 0 & 0 & 0 &0 & e^c \cos (\pi  k) \\
\end{array}
\right)
\end{equation}
Equation \ref{xrotpi} clearly shows that $x$ has precisely three fixed points as well.  Again in this case, $\dim H^{*}(x; \vv_{u}) = 3$ and similarly, we may fix $\theta$ and have one-dimension of freedom for $c$.\\
\\
For the final case, we consider when $c = 0$ and $\theta = \pi k$ for some $k \in \Z$.  The only two non-trivial choices depend on the parity of $k$ and are $c = 0$ and $\theta = 2j\pi$, in which $u = 0$, which corresponds to the parabolic case.  In this case $\dim H^{*}(x;\vv) = 3$, as clearly seen by the Jordan normal form of Equation \ref{PX}. \\
\\
If however we consider the odd multiple case where $c = 0$ and $\theta = (2j +1)\pi$, then our matrix representation for $\Ad \rho_{u}(x)$ becomes 
\begin{equation}\label{xelliptic}
\left(
\begin{array}{ccccccccc}
 1 & 0 & 0 & 0 & 0 & 0 & 0 & 0 & 0 \\
 0 & -1 & 0 & 0 & 0 & 0 & 0 & 0 & 0 \\
 0 & 0 & -1 & 0 & 0 & 0 & 0 & 0 & 0 \\
 0 & 0 & 0 & 1 & 0 & 0 & 0 & 0 & 0 \\
 0 & 0 & 0 & 0 & 1 & 0 & 0 & 0 & 0 \\
 0 & 0 & 0 & 0 & 0 & 1 & 0 & 0 & 0 \\
 0 & 0 & 0 & 0 & 0 & 0 & 1 & 0 & 0 \\
 0 & 0 & 0 & 0 & 0 & 0 & 0 & -1 & 0 \\
 0 & 0 & 0 & 0 & 0 & 0 & 0 & 0 & -1 \\
\end{array}
\right)
\end{equation}
which has five independent fixed points, and thus $\dim H^{*}(x; \vv_{u}) = 5$.  Note both of these cases are discrete without any degrees of freedom.\\
\\
As these cases exhaust all possible options in the hyperbolic Dehn filling space, we have proven the claim regarding the dimensions of each $H^{*}(x; \vv_{u})$.  The topological claims about the dimensions of the set of all representations of a fixed dimension follow from how we partitioned the cases.  In particular, $u = 0$, is the intersection of the two topological curves $\{u = i\theta \in U  |  |\theta| < \epsilon \}$ with $\{u = c \in U  |  |c| < \epsilon \}$.   
\end{proof}
\end{subsection}

\begin{subsection}{Fiber Bundle Cohomology}
Recall that the figure-eight knot complement $M$ is a once-punctured torus bundle over a circle.  Topologically, we have $F \lra M \lra S^{1}$ where $F$ is a once-punctured torus.  Instead of using the presentation in Equation \ref{wirt}, we may instead consider the fiber-bundle presentation given by
\begin{equation}\label{fibpres}
\pi_{1}(M) = \langle a,b,x \, | \, xax^{-1} = ba \text{ and } xbx^{-1} = bab \rangle
\end{equation}
This presentation can be recovered from the relation in Equation \ref{wirt} given by $ax = ya$.  The elements $a,b$ generate a free group of order two topologically corresponding to the once-punctured torus.  The quotient $\pi_{1}(M)/\langle a,b \rangle$ is cyclic and generated by the coset $\overline{x}$ for example.  The group theoretic analogue of the short exact sequence of topological spaces is
\begin{equation*}
1 \lra \langle a,b \rangle \lra \pi_{1}(M) \lra \langle \overline{x} \rangle \lra 1
\end{equation*}
Let $u \in U$ be in hyperbolic Dehn filling space, and consider the representation $\Ad \rho_{u} \de \pi_{1}(M) \lra \Aut(\vv)$.  Let $\vv_{u}$ denote the associated $\pi_{1}(M)$-module to this representation.  We invoke the Inflation-Restriction sequence via the Lydon-Hochschild-Serre spectral sequence to obtain a short exact sequence \cite[Section VII.6]{Brown1982Cohomology}.  
\begin{equation}\label{InfResSeq}
0 \lra H^{1}(\overline{x}; \vv_{u}^{F_{2}}) \lra H^{1}(\pi_{1}(M); \vv_{u}) \lra H^{1}(F_{2}; \vv_{u})^{\overline{x}} \lra H^{2}(\overline{x}; \vv_{u}^{F_{2}}) \lra \hdots
\end{equation}
Because $\overline{x}$ is the fundamental group of a circle, it has the homotopy type of a 1-complex, and so $H^{*}(\overline{x}; \vv_{u})$ vanishes in dimensions larger than 1.  Additionally, $\vv_{u}^{F_{2}}$ corresponds to the fixed points of the adjoint representation restricted to $\rho_{u}(F_{2})$.  Provided that $u \in U$ corresponds to a hyperbolic structure with cone or Dehn surgery type singularities, the fixed points are trivial so $\vv_{u}^{F_{2}} = 0$ \cite[Section 3.15]{Hodgson}.  Thus, Equation \ref{InfResSeq} reduces to
\begin{equation}\label{RedSeq}
0  \lra H^{1}(\pi_{1}(M); \vv_{u}) \lra H^{1}(F_{2}; \vv_{u})^{\overline{x}} \lra 0
\end{equation}
which means that $H^{1}(\pi_{1}(M);\vv_{u})$ is isomorphic to the fixed points of $H^{1}(F_{2}; \vv_{u})$ being acted upon by $\overline{x}$.  \\
\\
Before proceeding we briefly introduce relevant notation.  Let $\rho\de G \lra \Aut(V)$ be a representation of a group $G$ over a finite-dimensional vector space $V$.  For each subgroup $H \leq G$, we obtain a natural group homomorphism from $H^{*}(G; V) \lra H^{*}(H; V)$ simply by restriction.  We denote this map by $\res_{H}$.  The topological analogue arises from the induced map on cohomology via the inclusion map.  In the case where $H$ is cyclic and generated by $h$, we frequently denote this map by $\res_{h}$.  With the appropriate notation we may now state a result of Equation \ref{RedSeq}.
\begin{lemma}\label{injF2}
Let $i: F \lra M$ denote the embedding of a once-punctured torus arising from the fiber bundle $F \lra M \lra S^{1}$.  Then for any $u \in U$ in hyperbolic Dehn surgery space, the induced map on cohomology $H^{1}(M;\vv_{u}) \lra H^{1}(F;\vv_{u})$ is injective.  Moreover, the subset of all $u \in U$ for which $\dim H^{1}(M; \vv_{u}) = 1$ is an open subset.  
\end{lemma}
\begin{proof}
Equation \ref{RedSeq} holds for any $u \in U$ in hyperbolic Dehn filling space, thus we obtain that $\res_{F}: H^{1}(M;\vv_{u}) \lra H^{1}(F; \vv_{u})$ is injective.  The fact that the set of all $u \in U$ for which $H^{1}(M; \vv_{u})$ is one-dimensional is an open subset of hyperbolic Dehn filling space follows because $H^{1}(M; \vv_{u})$ is non-zero by the `Half-Lives Half-Dies Theorem' \cite[Corollary 5.4]{Heusener2011Infinitesimal}.  However, having expressed $H^{1}(M; \vv_{u})$ as the fixed point set of $H^{1}(F; \vv_{u})$ under the action of $\overline{x}$ on this 9-dimensional vector space means by semi-continuity the dimension of $H^{1}(M; \vv_{u})$ is non-increasing.  Thus for every point $u \in U$ for which $\dim H^{1}(M; \vv_{u}) = 1$, there is a neighborhood about $u$ for which $\dim H^{1}(M; \vv_{u})$ is one-dimensional.  
\end{proof}
One can interpret the lemma in the following manner.  If we consider a deformation of the representation $\Ad \rho_{u}: \pi_{1}(M) \lra \Aut(\vv)$ and its corresponding element in $H^{1}(M;\vv_{u})$, then Lemma \ref{injF2} states that one cannot deform $\Ad \rho_{u}$ without deforming the representation of the fundamental group $\pi_{1}(F)$ generated by the once-punctured torus.  This is the analogue of projective rigidity defined for $\del M$, but applied to the fiber of the sequence $F \lra M \lra S^{1}$.   
\end{subsection}

\begin{subsection}{Induced Action on the Once-Punctured Torus}\label{oncepuncturedaction}
In this section we calculate the induced action of $\overline{x}$ on $H^{1}(F; \vv_{u})$ where $\vv_{u}$ is the $\pi_{1}(M)$-module associated to the representation $\Ad \rho_{u}: \pi_{1}(M) \lra \Aut(\vv)$.  Throughout this section $M$ will denote the figure-eight knot complement, and $F$ will denote a once-punctured torus fiber whose fundamental group is the free group of rank 2, generated by $a$ and $b$.  Here we use the presentation of $\pi_{1}(M)$ given by Equation \ref{fibpres}.  \\
\\
We first consider $H^{1}(F_{2}; \vv_{u})$.  Because the universal cover of $F$ is contractible, we may relate the cohomology of $H^{*}(F_{2}; \vv_{u})$ with the topological cohomology of $H^{*}(F; \vv_{u})$ where $\vv$ is twisted by the representation $\Ad \rho_{u}$.  Because $F$ is homotopically equivalent to the wedge of two circles, we may use the topological chain-complex obtained from the universal cover of $S^{1}\vee S^{1}$.  Its universal cover is the Cayley graph of $F_{2}$, the infinite 4-valent tree whose edges are generated by $E = \{a,a^{-1},b,b^{-1}\}$ and whose vertices are generated by $1$.  As complexes, these are given by  
\begin{equation}\label{chaincomplexwedge}
0 \lra \Z[F_{2} \times S] \xrightarrow{\del} \Z[F_{2}]  \lra 0
\end{equation}
where $F_{2}$ acts on $F_{2}\times S$ via $h(g,s) = (hg,s)$ and where the boundary map acts on a basis element by $\del(g,s) = g - gs$.  Straight forward calculation shows that $\del$ is $F_{2}$-equivariant.  Applying the functor $\Hom_{\Z[F_{2}]}(\bullet, V)$ to Equation \ref{chaincomplexwedge} yields the sequence
\begin{equation}\label{cochaincomplexwedge}
0 \lra \Hom_{\Z[F_{2}]}\left(\Z[F_{2}], \vv_{u} \right) \xrightarrow{\delta} \Hom_{\Z[F_{2}]}\left(\Z[F_{2}\times S], \vv_{u} \right)   \lra 0
\end{equation}
Note that each $\Z[F_{2}]$-map from $\Z[F_{2}]$ into $\vv_{u}$ is determined by where $1$ is sent to, thus up to isomorphism, the first term of Equation \ref{cochaincomplexwedge} is isomorphic to $\vv_{u}$ as $\Z[F_{2}]$-modules.  Similarly each $\Z[F_{2}]$-map from $\Z[F_{2}\times S]$ into $\vv_{u}$ is determined by where $(1,a)$ and $(1,b)$ are sent to, and thus the second term is isomorphic to $\vv_{u} \times \vv_{u}$ as $\Z[F_{2}]$-modules.  The boundary map $\del$ dualizes to $\delta$ taking $\delta(f)(g):=f(\del (g))$ where $f \in \Hom_{\Z[F_{2}]}(\Z[F_{2}], \vv_{u})$.  As it is determined by its action on $(1, a)$ and $(1, b)$, we see 
\[
\delta(f)(1, a) = f(\del(1, a)) = f(1-a) = (1-a)f(1)
\]
and similarly so for $\delta(f)$ applied to $(1, b)$.  Substituting the relevant isomorphisms we see Equation \ref{chaincomplexwedge} becomes
\begin{equation}\label{cochainwedgeV}
0 \lra \vv_{u} \xrightarrow{i} \vv_{u} \times \vv_{u} \lra 0 \text{ where }i(v) = \left((1-a)v,(1-b)v\right)
\end{equation}
Provided the map $i$ as defined in Equation \ref{cochainwedgeV} is injective, as is the case when $u \in U$, we then have that $H^{1}(F; \vv_{u})$ is 9-dimensional. \\
\\
We conclude this section by exploring the induced action of $\overline{x}$ acting on $F$ where $F$ is the once-punctured torus fiber of the bundle $F \lra M \lra S^{1}$.  Recall the presentation of $\pi_{1}(M)$ as presented in Equation \ref{fibpres}.  As done in Section \ref{subsectwistedgroup}, we identify $H^{1}(F_{2};\vv_{u})$ as the crossed homomorphisms $f\de F_{2} \lra \vv$ satisfying $f(uv) = f(u) + uf(v)$ for all $u, v \in F_{2}$, modulo the coboundaries.
\\
\\
Because $F_{2}$ has no relations, such a crossed homomorphism is specified by a single pair of vectors $(v,w) \in \vv\times \vv$.  These vectors correspond to the crossed homomorphisms values on the generators $a$ and $b$ of $F_{2}$.  The group $\pi_{1}(M)$ acts on these crossed homomorphisms in the following manner.  Let $f(a) = v$ and $f(b) = w$.  The element $g \in \pi_{1}(M)$ acts on $F_{2}$ via conjugation, as $F_{2}$ is normal.  We define a new crossed homomorphism $gf:\pi_{1}(M) \lra \vv$ via $(gf)(h) = gf(g^{-1}hg)$ for each $h \in F_{2}$.   \\
\\
One can check by routine calculation this action preserves the coboundaries, and thus descends to an action on $H^{1}(F_{2}; \vv_{u})$.  What is perhaps less obvious is that the action restricted to $F_{2}$ on $H^{1}(F_{2}; \vv_{u})$ is trivial.  This is because in the conjugation relation, if both $g, h \in F_{2}$, then one can use the crossed homomorphism properties to expand $f(g^{-1}hg)$ as 
\begin{align*}
(gf)(h) &= gf(g^{-1}hg) = g\left(f(g^{-1}) + g^{-1}f(hg)\right) = g\left(-g^{-1}f(g) + g^{-1}f(h) + g^{-1}hf(g)\right)\\
&=f(h) + (h-1)f(g)
\end{align*}
In the quotient we see that $gf = f \in H^{1}(F_{2}; \vv_{u})$ for all $g \in F_{2}$.  Thus, the induced action of $\pi_{1}(M)/F_{2} = \langle \overline{x} \rangle$ is the only non-trivial part.  We pick $x \in \pi_{1}(M)$ as a representative of the quotient and investigate its action on $\vv\times \vv$.  \\
\\
Recall the conjugation relations as in Equation \ref{fibpres} which can be shown to obey the equations below.
\begin{equation*}
x^{-1}ax = b^{-1}a^{2} \text{ and } x^{-1}bx = a^{-1}b
\end{equation*}
Expanding these terms via the crossed homomorphism rules gives us that
\begin{equation}\label{inducedaction}
(xf)(a) = (xb^{-1}+xb^{-1}a)f(a) - xb^{-1}f(b) \text{ and }(xf)(b) = -xa^{-1}f(a) + xa^{-1}f(b)
\end{equation}
Thus we have an induced action of $x$ on $\vv^{2}$.  We denote this $x$-module by $\vv_{u}^{2}$.  If we identify the cyclic groups $x$ and $\overline{x}$ via the canonical projection $\pi_{1}(M) \lra \pi_{1}(M)/F_{2}$, then Equation \ref{cochainwedgeV} becomes a short exact sequence of $x$-modules
\begin{equation}\label{sesv2}
\vv_{u} \xrightarrow{\phantom{=}i\phantom{=}} \vv_{u}^{2} \lra H^{1}(F_{2};\vv_{u}) \text{ where } i(v) = \left((1-a)v, (1-b)v \right)
\end{equation}
where the $x$-action on $\vv_{u}$ is the one restricted from $\pi_{1}(M)$ to the cyclic group generated by $x$, the $x$-action on $\vv_{u}^{2}$ is the action calculated in Equation \ref{inducedaction}, and the $x$-action on $H^{1}(F ;\vv_{u})$ is the quotient $x$-module.  We remark that the action of $\overline{x}$ on $H^{1}(F_{2};\vv_{u})$ determines the dimension of the cohomology group $H^{1}(M;\vv_{u})$ due to the isomorphism afforded in Equation \ref{RedSeq}.  Specifically, $\dim H^{1}(M; \vv_{u}) = \dim \ker (1-x):H^{1}(F_{2}; \vv_{u}) \lra H^{1}(F_{2}; \vv_{u})$; the number of distinct fixed directions of the action of $x$ on $H^{1}(F_{2}; \vv_{u})$.  \\
\\
We summarize these results with Lemma \ref{injF2} in the following theorem.
\begin{theorem}\label{thmdimeq}
Let $u \in U$ be a point in hyperbolic Dehn surgery space of the figure-eight knot complement $M$ and let $\vv_{u}$ denote the corresponding $\pi_{1}(M)$-module.  
Let $F$ denote a once-punctured torus fiber.  Then for any $u \in U$, we have that $\dim H^{1}(F; \vv_{u}) = 9$ and $H^{1}(M; \vv_{u})$ is isomorphic to $H^{1}(F_{2}; \vv_{u})^{x}$ where $x$ is the monodromy map of the fiber bundle.  Thus $\dim H^{1}(M; \vv_{u}) = 1$ if and only if there is a unique fixed line of the induced monodromy map of $x$ acting on $H^{1}(F_{2}; \vv_{u})$ and the subset of all such $u \in U$ for which this is true is open.
\end{theorem}
\end{subsection}

\begin{subsection}{Cohomology of the Once-Punctured Torus}\label{oncepuncturedtorus}
To ease notation, we will occasionally suppress the coefficient module $\vv_{u}$ when convenient.  Let $F$ denote the once-punctured torus fiber of the bundle $F \lra M \lra S^{1}$.  Let $\del F$ denote the boundary of $F$ generated by the longitude $l = b^{-1}a^{-1}ba$.  Let $U$ denote hyperbolic Dehn filling space of $M$.  \\
\\
With this notation established, consider the inclusion map $\del F \subset F$, and the induced maps on cohomology $\res_{\del F}: H^{*}(F) \lra H^{*}(\del F)$.  We obtain a long exact sequence in twisted cohomology given by
\begin{equation}
0 \lra H^{0}(F;\del F) \lra H^{0}(F) \lra H^{0}(\del F) \lra H^{1}(F; \del F) \lra H^{1}(F) \lra H^{1}(\del F) \lra 0
\end{equation}
The higher order groups vanish due to the fact that $F$ and $\del F$ are homotopic to one-dimensional complexes.  Because $H^{0}(F) = 0$ as $F_{2}$ fixes no non-zero points in $\vv_{u}$, we may reduce the above sequence to
\begin{equation}\label{lesgp}
0 \lra H^{0}(\del F) \lra H^{1}(F; \del F) \lra H^{1}(F) \lra H^{1}(\del F) \lra 0
\end{equation}
As both $F$ and $\del F$ are preserved by the action of conjugation of $x$, Equation \ref{lesgp} is a long exact sequence of $x$-modules.  This section will be dedicated to understanding some properties of the above sequence.  We begin by first proving the following.
\begin{lemma}\label{dualmodules}
For any point $u \in U$, the $x$-modules $H^{0}(\del F; \vv_{u})$ and $H^{1}(\del F; \vv_{u})$ are isomorphic.  Moreover, $\dim H^{0}(\del F; \vv_{u})^{x} = \dim H^{1}(\del F; \vv_{u})^{x} = 1$.  
\end{lemma}
\begin{proof}
The statement can be reduced to analyzing the $x$-action on $H^{*}(\del F; \vv_{u}) = H^{*}(l; \vv_{u})$.  For $u \neq 0$, up to conjugation, the $l$-action on $\vv_{u}$ is of the form of either Equation \ref{xhyplengthzero}, Equation \ref{xrotpi}, or Equation \ref{xelliptic}.  In each case we have the decomposition
\begin{equation}\label{vectordecomp}
\vv_{u} = \ker (1-\Ad \rho_{u}(l)) \oplus \im (1-\Ad \rho_{u}(l))
\end{equation}
Equation \ref{vectordecomp} is an equality of $x$-modules, as $x$ commutes with $l$.  By Lemma \ref{circcohom}, $H^{0}(l; \vv_{u}) = \vv_{u}^{l}$ and $H^{1}(l; \vv_{u}) = \left(\vv_{u}\right)_{l}$, where $\left(\vv_{u}\right)_{l}$ denotes the co-invariants of $\Ad \rho_{u}(l)$ acting on $\vv$.  \\
\\
We may use the decomposition in Equation \ref{vectordecomp} to see that $ \left(\vv_{u}\right)_{l}$ is isomorphic to $\ker (1-\Ad \rho_{u}(l)) = \vv_{u}^{l}$ as $x$-modules.  This shows that for non-zero $u \in U$, we have as $x$-modules that $H^{0}(\del F;\vv_{u})$ is isomorphic to $H^{1}(\del F; \vv_{u})$.    \\
\\
In the case where $u = 0$, we lose the equality of Equation \ref{vectordecomp} as the kernel and image intersect non-trivially.  We may explicitly write out the $x$-action on $H^{0}(l;\vv_{0})$ and $H^{1}(l;\vv_{0})$ and compare the two.  The Jordan form of $\Ad \rho_{0}(l)$ acting on $\vv$ is given by
\begin{equation*}\label{parabolicl}
\Ad \rho_{0}(l) = \left(
\begin{array}{ccccccccc}
 1 & 0 & 0 & 0 & 0 & 0 & 0 & 0 & 0 \\
 0 & 1 & 1 & 0 & 0 & 0 & 0 & 0 & 0 \\
 0 & 0 & 1 & 1 & 0 & 0 & 0 & 0 & 0 \\
 0 & 0 & 0 & 1 & 0 & 0 & 0 & 0 & 0 \\
 0 & 0 & 0 & 0 & 1 & 1 & 0 & 0 & 0 \\
 0 & 0 & 0 & 0 & 0 & 1 & 1 & 0 & 0 \\
 0 & 0 & 0 & 0 & 0 & 0 & 1 & 1 & 0 \\
 0 & 0 & 0 & 0 & 0 & 0 & 0 & 1 & 1 \\
 0 & 0 & 0 & 0 & 0 & 0 & 0 & 0 & 1 \\
\end{array}
\right)
\end{equation*}
which has three Jordan-blocks of size $1,3$, and $5$ all with eigenvalues $1$; compare to Equation \ref{PX}.  It is clear that both $H^{0}(l;\vv_{0})$ and $H^{1}(l;\vv_{0})$ are 3-dimensional.  We may identify $H^{0}(l; \vv)$ with the space spanned by $e_{1},e_{2},e_{5}$, whereas we may identify $H^{1}(l;\vv_{0})$ with the images of $e_{1}, e_{4}, e_{9}$ under the projection $\R^{9} \lra \R^{9}/\im (1-\Ad \rho_{0}(l))$.  \\
\\
As $x$ preserves both $H^{0}(l;\vv_{0})$ and $H^{1}(l;\vv_{0})$, we may use the bases in the paragraph described above to yield the $x$-action on $H^{0}(l;\vv_{0})$ and $H^{1}(l;\vv_{0})$ as
\begin{equation}\label{parabolicxinducedactions}
\rho_{0}(x)|_{H^{0}(l;\vv_{0})} = \left(
\begin{array}{ccc}
 1 & 0 & 0 \\
 \frac{4}{3} & 1 & 0 \\
 -\frac{1}{108} & -\frac{1}{72} & 1 \\
\end{array}
\right) \text{ and } 
\rho_{0}(x)|_{H^{1}(l;\vv_{0})} = \left(
\begin{array}{ccc}
 1 & \frac{1}{6} & -\frac{1}{6} \\
 0 & 1 & -2 \\
 0 & 0 & 1 \\
\end{array}
\right)
\end{equation}
A straight forward calculation shows that both matrices in Equation \ref{parabolicxinducedactions} have the same Jordan canonical form, namely 1 block of size 3, and thus, as $x$-modules $H^{0}(l;\vv_{0})$ and $H^{1}(l;\vv_{0})$ are isomorophic.  This proves that for $u \in U$, we have $H^{0}(\del F; \vv_{u})$ and $H^{1}(\del F; \vv_{u})$ are isomorphic as $x$-modules.\\
\\
We now need only to calculate the dimension of $H^{0}(\del F; \vv_{u})^{x}$ for all $u \in U$.  We remark the remainder of the proof provided here is a very minor deviation from \cite[Lemma 5.1]{Heusener2011Infinitesimal} and provided largely in an effort for self-containment.  Observe that $H^{0}(\del F; \vv_{u})^{x}$ is simply the subspace of $\vv$ that is fixed by both $\Ad \rho_{u}(x)$ and $\Ad \rho_{u}(l)$, thus it suffices to calculate the dimension of the fixed points of the boundary torus $\pi_{1}(T^{2}) < \pi_{1}(M)$.  We proceed by cases.  For $u = 0$, we can see clearly from Equation \ref{parabolicxinducedactions} that for the action of $x$ on $H^{0}(l; \vv_{0})$, there is a single fixed point, so $\dim H^{*}(\del F; \vv_{0})^{x} = 1$.  \\
\\
For $u \neq 0$, we have that $\rho_{u}(l)$ is loxodromic and of the form of Equation \ref{xactioneig}.  If $u$ has non-integer multiple of $\pi$-rotational part and non-zero translational part, then up to conjugation $\Ad \rho_{u}(l)$ and $\Ad \rho_{u}(x)$ are of the form in Equation \ref{psorep}, which has a one-dimensional infinitesimal centralizer given by
\begin{equation*}
\left(
\begin{array}{cccc}
 1 & 0 & 0 & 0 \\
 0 & 1 & 0 & 0 \\
 0 & 0 & -1 & 0 \\
 0 & 0 & 0 & -1 \\
\end{array}
\right)
\end{equation*}
If however say for example $\rho_{u}(l)$ has trivial translation, then $\Ad \rho_{u}(l)$ has the form of Equation \ref{xhyplengthzero}.  This means representation of $\rho_{u}(x)$ cannot also be of the same form as in Equation \ref{xhyplengthzero}, as this would imply that in $H^{3}$, the subgroup generated by the boundary torus of $x$ and $l$ preserves a common geodesic upon which both act by rotation which would contradict the fact that $x$ and $l$ generated the fundamental group of a torus.  Thus, $\rho_{u}(x)$ must have non-trivial translational part.  Consequently $\Ad \rho_{u}(x)$ must be of the form Equation \ref{xrotpi}.  Inspection of Equations \ref{xhyplengthzero} and \ref{xrotpi} yield a single fixed vector.  An analogous arguments holds if $\rho_{u}(l)$ has trivial rotational part.  Thus for $u \in U$, $\dim H^{0}(\del F; \vv) = \dim H^{1}(\del F; \vv) = 1$
\end{proof}
With Lemma \ref{dualmodules} established, we may state the following theorem. 
\begin{theorem}\label{Fcohom}
For each $u \in U$, let $I_{u}$ denote the image of the map between $x$-modules $H^{1}(F,\del F; \vv_{u}) \lra H^{1}(F; \vv_{u})$ as in the long exact sequence from Equation \ref{lesgp}.  Then 
\begin{equation}\label{flses}
I_{u} \lra H^{1}(F; \vv_{u}) \lra H^{1}(\del F; \vv_{u})
\end{equation}
is a short exact sequence of $x$-modules.  Moreover, for a sufficiently small neighborhood $V \subset U$ about $0$, we have that $H^{1}(F; \vv_{v})^{x}$ is isomorphic to $H^{1}(l; \vv_{v})^{x}$.  In particular, for $v \in V$, we have that $\dim H^{1}(F;\vv_{v})^{x} = 1$ for all $v \in V$.  
\end{theorem}
\begin{proof}
For parts of this proof, we suppress the coefficient module $\vv_{u}$ unless we wish to draw particular attention to a specific representation.  From the exact sequence of $x$-modules in Equation \ref{lesgp}, we obtain the short exact sequence of $x$-modules in Equation \ref{flses}.  We may thus consider the associated sequence of group cohomology of the cyclic group $\langle x \rangle$.  This yields
\begin{equation}\label{exseqcirclecohom}
I_{u}^{x} \lra H^{1}(F)^{x} \lra H^{1}(\del F)^{x} \lra (I_{u})_{x} \lra H^{1}(F)_{x} \lra H^{1}(\del F)_{x}
\end{equation}
To prove that for a sufficiently small neighborhood $V$ about 0 that $H^{1}(F; \vv_{v})^{x}$ is one-dimensional, we need only to consider the case at the parabolic point $v = 0$ by Theorem \ref{thmdimeq}. \\
\\
Consider Equation \ref{exseqcirclecohom} with $v = 0$.  If we manage to show that $I_{0}^{x} = 0$, then $H^{1}(F; \vv_{0})^{x}$ will be isomorphic to $H^{1}(\del F; \vv_{0})^{x}$, the latter of which we know is one-dimensional by Lemma \ref{dualmodules}.  By Equation \ref{flses}, we have that $H^{1}(F; \vv_{0})$ is an extension of $H^{1}(\del F; \vv_{0})$ by $I_{0}$.  Thus the characteristic polynomial of $x$ acting on $H^{1}(F; \vv_{0})$ is the product of the characteristic polynomials of $x$ acting on $H^{1}(\del F; \vv_{0})$ and $I_{0}$.  By the analysis of Lemma \ref{dualmodules}, specifically Equation \ref{parabolicxinducedactions}, If we manage to show that the characteristic polynomial of $H^{1}(F;\vv_{0})$ factors as $(1-t)^{3}q(t)$ with $q(1) \neq 0$, then we are done.  \\
\\
To calculate the characteristic polynomial of the $x$-module $H^{1}(F;\vv_{0})$, we instead calculate the characteristic polynomial of the $x$-module $H^{1}(F;\del F; \vv_{0})$ which are equal by Poincar\'{e} Duality.  In \texttt{PRFE\textunderscore Mathematica.nb} we see that the characteristic polynomial of the latter is given by
\begin{equation}\label{xh1fpoly}
-(t-1)^3 \left(t^6-15 t^5+27 t^4-42 t^3+27 t^2-15 t+1\right)
\end{equation}
In the notebook, this characteristic polynomial is labeled \texttt{charxPH1FDF}.  The term $q(t)$ in Equation \ref{xh1fpoly} evaluates to a non-zero quantity when $t = 1$, and thus $I_{0}$ admits no additional fixed points of $x$, hence $I_{0}^{x} = 0$ as claimed.  
\end{proof}
As a corollary of this, we end up with a rigidity result regarding the map $H^{1}(F) \lra H^{1}(\del F)$.  Compare this result to that of Ballas's in Section \ref{ssballas}.  
\begin{corollary}\label{deforml}
Let $u \in V$ correspond to a representation sufficiently close to $0 \in V$ for which $H^{1}(F; \vv_{u})^{x}$ is one-dimensional as in Theorem \ref{Fcohom}.  Then there exists a sufficiently small neighborhood $W \subset V$ about $0$ for which the map $H^{1}(F; \vv_{u}) \lra H^{1}(\del F; \vv_{u})$ is non-zero for all $u \in W$.  In addition $I_{u} = 0$ for all $u \in W$.  
\end{corollary}
\begin{proof}
Picking a $u \in V$ guarantees that $H^{1}(F; \vv_{u})^{x}$ is locally one-dimensional.  Lemma \ref{dualmodules} guarantees the same of $H^{1}(\del F; \vv_{u})^{x}$.  Thus the map $H^{1}(F; \vv_{u})^{x} \lra H^{1}(\del F; \vv_{u})^{x}$ in Equation \ref{exseqcirclecohom} is to and from one-dimensional spaces.  If we prove that the map $H^{1}(F; \vv_{0})^{x} \lra H^{1}(\del F; \vv_{0})^{x}$ is an isomorphism, then the same will be true for the maps $H^{1}(F; \vv_{u})^{x} \lra H^{1}(\del F; \vv_{u})^{x}$ for each $u$ within a sufficiently small neighborhood $W$ containing $0$ by semi-continuity. \\
\\
Let us simply find an element $f \in H^{1}(M; \vv_{u})$ with non-zero image.  If we define the crossed homomorphism $f: \pi_{1}(M) \lra \vv$ via
\begin{equation}\label{Pcohom}
f(x) = 0 \text{ and } f(y) = \left(
\begin{array}{cccc}
 14 & -31 \sqrt{3} & -57 & 81 \\
 -31 \sqrt{3} & 144 & 91 \sqrt{3} & -115 \sqrt{3} \\
 -57 & 91 \sqrt{3} & 181 & -260 \\
 -81 & 115 \sqrt{3} & 260 & -339 \\
\end{array}
\right)
\end{equation}
we can readily verify that $f$ is not a coboundary as $(f(x),f(y))$ lies outside the span of $(1-\Ad \rho_{0} (x),1-\Ad \rho_{0} (y))$ on $\vv_{0}\times \vv_{0}$.  In fact, if we calculate $f$ restricted to $\pi_{1}(F)$, we see that
\begin{equation}\label{Pcohomresab}
f(a) = \left(
\begin{array}{cccc}
 40 & 20 \sqrt{3} & -\frac{297}{2} & \frac{307}{2} \\
 20 \sqrt{3} & 0 & -51 \sqrt{3} & 63 \sqrt{3} \\
 -\frac{297}{2} & -51 \sqrt{3} & 638 & -670 \\
 -\frac{307}{2} & -63 \sqrt{3} & 670 & -678 \\
\end{array}
\right) \text{ and } f(b) = \left(
\begin{array}{cccc}
 -88 & -24 \sqrt{3} & \frac{285}{4} & -\frac{475}{4} \\
 -24 \sqrt{3} & 0 & \frac{35 \sqrt{3}}{4} & -\frac{21 \sqrt{3}}{4} \\
 \frac{285}{4} & \frac{35 \sqrt{3}}{4} & -\frac{137}{4} & \frac{231}{4} \\
 \frac{475}{4} & \frac{21 \sqrt{3}}{4} & -\frac{231}{4} & \frac{489}{4} \\
\end{array}
\right)
\end{equation}
One may readily verify that $f$ is invariant under the induced action of $x$, thus descends to a fixed point in $H^{1}(F; \vv_{0})$ under $x$.  A straight forward calculation using the properties of crossed homomorphisms yields that
\begin{equation}\label{Ponl}
f(l) = f(b^{-1}a^{-1}ba) = \left(
\begin{array}{cccc}
 -40 & 0 & 0 & 0 \\
 0 & 120 & -80 \sqrt{3} & -80 \sqrt{3} \\
 0 & -80 \sqrt{3} & 151 & 191 \\
 0 & 80 \sqrt{3} & -191 & -231 \\
\end{array}
\right)
\end{equation}
which lies outside the image of $1-\Ad \rho_{0}(l)$, and thus the map $H^{1}(F; \vv_{0})^{x} \lra H^{1}(\del F; \vv_{0})^{x}$ is an isomorphism.  These calculations are carried out in detail in \texttt{PRFE\textunderscore Mathematica.nb}.
\end{proof}
We remark that in the proof of Corollary \ref{deforml} we also showed the map $\res_{x}: H^{1}(M; \vv_{0}) \lra H^{1}(x; \vv_{0})$ is trivial.  One can interpret this as the statement that it is possible, at the ideal point, to deform the hyperbolic structure of $M$ while keeping the projective structure on the circle generated by $x$ fixed.  However, Corollary \ref{deforml} says that non-trivial deformations of the projective structure on $M$ lead to non-trivial deformations of $\del F$.  This is an enhancement of the notion of projective rigidity relative to the boundary, which states that non-trivial deformations of $M$ lead to non-trivial deformations of $\del M$, as $\del F \subset \del M$.  
\end{subsection}

\begin{subsection}{Slope of a Representation}
Let $M$ be the figure-eight knot complement and $x$ and $l$ be a meridian and longitude of $\pi_{1}(M)$.  Let $u \in U$ be a point in hyperbolic Dehn filling space, and $\vv_{u}$ the associated $\pi_{1}(M)$-module to $\Ad \rho_{u}$.  Let us assume further that $\dim H^{1}(M; \vv_{u}) = 1$.  Then by Theorem \ref{thmdimeq}, there's an open subset $V$ about $u$ for which $\dim H^{1}(M; \vv_{v}) = 1$ for all $v \in V$.  Let us assume further that at this $u$, we have both restriction maps $\res_{x}\de H^{1}(M; \vv_{u}) \lra H^{1}(x; \vv_{u})$ and $\res_{l}: H^{1}(M; \vv_{u}) \lra H^{1}(l; \vv_{u})$ are isomorphisms.  We may shrink $V$ about $u$ to combine these conditions so that $\res_{x}\circ \res_{l}^{-1}\de H^{1}(l; \vv_{v}) \lra H^{1}(x; \vv_{v})$ is an isomorphism and $\dim H^{1}(M; \vv_{v}) = 1$ for all $v \in V$.  By Theorem \ref{thmdimeq} and Lemma \ref{circdim}, the set of all such $u$ for which these conditions hold is an open subset of hyperbolic Dehn filling space.\\
\\
Let $u \in U$ be such a point in hyperbolic Dehn filling space.  By Equation \ref{vectordecomp}, $\vv_{u} = \ker(1-\Ad \rho_{u}(x))\oplus \im(1-\Ad \rho_{u}(x))$, so we may naturally identify $H^{1}(x; \vv_{u})$ with $H^{0}(x; \vv_{u}) = \vv_{u}^{x}$.  A similar statement may be said about identifying $H^{1}(l; \vv_{u})$ and $H^{0}(l; \vv_{u}) = \vv_{u}^{l}$.  By Lemma \ref{dualmodules}, we have that $\vv_{u}^{x} =  \vv_{u}^{l}$ and thus by Poincar\'{e} Duality, $\dim H^{1}(T^{2}; \vv_{u}) = 2$.  Combining the restriction isomorphisms $\res_{x}$ and $\res_{l}$, we obtain an embedding of $H^{1}(T^{2}; \vv_{u})$ into $\vv\oplus \vv$ as a two-dimensional subspace.  
\begin{equation*}\label{torusiso}
H^{1}(T^{2}; \vv_{u}) \lra \vv_{u}^{x} \oplus \vv_{u}^{l} \lra \vv \oplus \vv
\end{equation*}
Thus, for all points of hyperbolic Dehn filling space sufficiently close to $u$ inside of $V$, we may consider the composition
\begin{equation}\label{composition}
H^{1}(M; \vv_{v}) \lra H^{1}(T^{2}; \vv_{v}) \lra \vv \oplus \vv
\end{equation}
which is an isomorphism onto a one-dimensional subspace of $\vv_{v}^{x} \oplus \vv_{v}^{l}$.  By these identifications we may speak unambiguously about $\res_{T^{2}}$ taking values in $\vv_{v}^{x} \oplus \vv_{v}^{l}$.\\
\\
Thus, we may pick a smoothly varying non-zero vector for the one-dimensional space $\vv_{u}^{x} = \vv_{u}^{l}$, call it $a_{u} \in \vv$, and a smoothly varying non-zero vector for the one-dimensional space $H^{1}(M;\vv_{u})$, call it $z_{u}$.  In \texttt{PRFE\textunderscore Mathematica.nb}, we explicitly calculate $a_{u}$, it is labeled as the variable \texttt{au}.  Relative to the basis $\left\{(a_{u},0),(0,a_{u}) \right\} \subset \vv_{u}^{x} \oplus \vv_{u}^{l}$, we know by the above remarks that the image of $\res_{T^{2}}(z_{u})$ is equal to $(c_{1u}a_{u} ,c_{2u}a_{u}) \in \vv_{u}^{x}\oplus \vv_{u}^{l}$ for smoothly varying non-zero $c_{1u}, c_{2u} \in \R$.  Neither $c_{1u}$ nor $c_{2u}$ can be zero as both restrictions of $\pi_{1}(M)$ to the circles $x$ and $l$ are isomorphisms on the level of first cohomology.  With this established, we provide a definition.
\begin{definition}\label{defslope}
Let $u \in U$ be a point in hyperbolic Dehn filling space and $\vv_{u}$ be the corresponding $\pi_{1}(M)$-module.  Assume further that at this point $u$, both $\res_{x}\de H^{1}(M; \vv_{u}) \lra H^{1}(x; \vv_{u})$ and $\res_{l}: H^{1}(M; \vv_{u}) \lra H^{1}(l; \vv_{u})$ are isomorphisms between one-dimensional vector spaces.  With the quantities as defined in the previous paragraph, we call $s_{u} := c_{1u}/c_{2u}$ the \emph{slope} of the representation $\rho_{u}\de \pi_{1}(M) \lra \Aut(\vv)$ for $u \in U$. 
\end{definition}
By construction this quantity varies analytically and is non-zero is some sufficiently small neighborhood of about $u$.  Another way of thinking about the slope is as follows.  Given a point $u \in U$ in hyperbolic Dehn filling space with $\dim H^{1}(M; \vv_{u}) = 1$ and for which the restriction maps onto each circle generated by $x$ and $l$ induce isomorphisms on the cohomology, we have an isomorphism $\res_{x}\circ\res_{l}^{-1}: H^{1}(l;\vv_{u}) \lra H^{1}(x; \vv_{u})$.  If we identify both $H^{1}(l; \vv_{u})$ and $H^{1}(x; \vv_{u})$ with $H^{0}(l;\vv_{u}) = \vv_{u}^{l}$ and $H^{0}(x;\vv_{u}) = \vv_{u}^{x}$, and consider the map $\res_{x}\circ\res_{l}^{-1}$ relative to the basis $a_{u}$ for both $ \vv_{u}^{l} = \vv_{u}^{x} $, we obtain $\res_{x}\circ\res_{l}^{-1}$ is simply scalar multiplication by a non-zero constant.  This constant is the slope of $\rho_{u}\de \pi_{1}(M) \lra \Aut(\vv_{u})$.  
\\
\\
In \cite[Section 6]{Heusener2011Infinitesimal}, Heusener and Porti define a function on hyperbolic Dehn filling space $U$ which we review briefly.  Let $u \in U$ be a point such that $\dim H^{1}(M; \vv_{u}) = 1$.  Choose a smooth non-vanishing vector $a_{u} \in H^{0}(T^{2}; \vv_{u})$ and a $z_{u} \in H^{1}(T^{2}; \vv_{u})$ which forms a basis for the image of $\res_{T^{2}}\de H^{1}(M; \vv_{u}) \lra H^{1}(T^{2}; \vv_{u})$.  Define a smooth non-vanishing element of the first-homology group by
\begin{equation*}\label{homologyelement}
a_{u}\otimes\left(p_{u}x + q_{u}l\right) \in H_{1}(\del M; \vv_{u})
\end{equation*}
where $p_{u}$ and $q_{u}$ are the generalized Dehn filling coefficients of the representation $\rho_{u}$ as defined in Section \ref{ssdehnfill}.  Let $\langle \cdot , \cdot \rangle$ denote the Kronecker pairing between homology and cohomology from Section \ref{cohompairings}.  In Remark 6.2 Heusener and Porti \cite{Heusener2011Infinitesimal} define
\begin{equation}\label{functionf}
f(u) = \langle z_{u}, a_{u}\otimes \left(p_{u}x + q_{u}l\right) \rangle = B\left(z_{u}\left(p_{u}x + q_{u}l\right), a_{u}\right)
\end{equation}
where $B$ denotes the Killing form on $\vv$ and $p_{u}, q_{u}$ are the generalized Dehn filling coefficients corresponding to the representation $u \in U$.  The second equality follows from Poincar\'{e} duality and viewing $z_{u}$ as a map on simplicial chains as in Section \ref{cohompairings}.  Here we are using a scaled version of the function as defined in \cite{Heusener2011Infinitesimal}, nevertheless, we are only interested in showing that $f$ is non-zero at certain points.  This is because of the following specialized case of \cite[Lemma 6.4]{Heusener2011Infinitesimal}.
\begin{lemma}\label{lemmafnonzero}
Let $(p_{u},q_{u})$ be generalized Dehn filling coefficients for a point $u \in U$ such that $\dim H^{1}(M; \vv_{u}) = 1$.  Let $p_{u}$ and $q_{u}$ be relatively prime integers and let $M_{p,q}$ denote the $(p,q)$ Dehn filling of $M$.  If $f(u) \neq 0$, then $H^{1}(M_{p,q},\vv_{u}) = 0$ 
\end{lemma}
For the remainder of this subsection, we aim to rephrase the language of Lemma \ref{lemmafnonzero} in the context of the slope of the representation.  Let $u \in U$ be a point for which both $\res_{x}\de H^{1}(M; \vv_{u}) \lra H^{1}(x; \vv_{u})$ and $\res_{l}: H^{1}(M; \vv_{u}) \lra H^{1}(l; \vv_{u})$ are isomorphisms between one-dimensional vector spaces.  As done in the previous paragraphs, we identify $H^{1}(x;\vv_{u})$ with $H^{0}(x; \vv_{u}) = \vv_{u}^{x}$ by Poincar\'{e} Duality, and the same for $l$.  Under this identification we may unambiguously speak of $\res_{x}(z_{u}) = z_{u}(x)$ and $\res_{l}(z_{u}) = z_{u}(l)$ as both taking values in $\vv_{u}^{x} = \vv_{u}^{l}$.  
\\
\\
We proceed by first analyzing the possible values of $z_{u}(x)$ and $z_{u}(l)$ where $z_{u}: \pi_{1}(T^{2}) \lra \vv$ is a crossed homomorphism.  Observe that
\begin{equation}\label{crossedcom}
0 = z_{u}(xlx^{-1}l^{-1}) = (1- xlx^{-1})z_{u}(x) + x(1-lx^{-1}l^{-1})z_{u}(l) = (1-l)z_{u}(x) + (x-1)z_{u}(l)
\end{equation}
Because both $l$ and $x$ have one-dimensional fixed axes in $\vv_{u}$ generated by $a_{u}$, $z_{u}(x)$ is parallel to $a_{u}$ if and only if $z_{u}(l)$ is parallel to $a_{u}$.  Additionally, both $\res_{x}(z_{u})$ and $\res_{l}(z_{u})$ are non-zero, as the restriction maps are isomorphisms.  By definition of the slope of the representation, we have that $z_{u}(x) = c_{1u}a_{u}$ and $z_{u}(l) = c_{2u}a_{u}$ with neither $c_{1u}$ nor $c_{2u}$ equal to zero.\\
\\
Equation \ref{functionf} then reduces to
\begin{equation*}
f(u) = B\left(z_{u}\left(p_{u}x + q_{u}l\right), a_{u}\right) = p_{u}c_{1u}B(a_{u},a_{u})+ q_{u}c_{2u}B(a_{u},a_{u}) 
\end{equation*}
Because $u \neq 0$,  $B(a_{u},a_{u}) \neq 0$ \cite[Lemma 5.3]{Heusener2011Infinitesimal}, and thus we have that $f(u) = 0$ if and only if $p_{u}(c_{1u}/c_{2u}) + q_{u} = 0$, which is to say that the slope $s_{u}$ of $\rho_{u}\de \pi_{1}(M) \lra \Aut(\vv)$ is equal to $-q_{u}/p_{u}$.  \\
\\
Combining the above results with Lemma \ref{lemmafnonzero}, we have the following Lemma.
\begin{lemma}\label{suffcond}
Let $u \in U$ be a point in hyperbolic Dehn filling space for which $\dim H^{1}(M; \vv_{u}) = 1$ and for which $\res_{x}: H^{1}(M; \vv_{u}) \lra H^{1}(x; \vv_{u})$ and $\res_{l}: H^{1}(M; \vv_{u}) \lra H^{1}(l; \vv_{u})$ are both isomorphisms.  Assume further that $u$'s generalized Dehn filling coefficients are relatively prime integers $(p,q)$.  If the slope of the representation $s_{u} \neq -q/p$, then the closed hyperbolic manifold obtained by Dehn filling along a slope of $p/q$, $M_{p,q}$, is infinitesimally projectively rigid, i.e. $H^{1}(M_{p,q},\vv_{u}) = 0$.  
\end{lemma}
\end{subsection}

\begin{subsection}{Some Sufficient Conditions}\label{ssssc}
In this subsection we aim to provide some sufficient conditions to satisfy parts of the hypotheses of Lemma \ref{suffcond}.  Specifically, we wish to have sufficient conditions that guarantee that both $H^{1}(M; \vv_{u})$ is 1-dimensional and the induced map $\res_{x}: H^{1}(M; \vv_{u}) \lra H^{1}(x;\vv_{u})$ is an isomorphism.  \\
\\
Let $u \in U$ be a point in hyperbolic Dehn filling space with $\pi_{1}(M)$-module $\vv_{u}$.  Consider the short exact sequence of $x$-modules as defined in Equation \ref{sesv2}.  This short exact-sequence gives rise to a long-exact sequence on cohomology of the cyclic group $\langle x \rangle$.  Explicitly, it is
\begin{equation}\label{exactxv2}
\vv_{u}^{x} \lra (\vv_{u}^{2})^{x} \lra H^{1}(F; \vv_{u})^{x} \lra (\vv_{u})_{x} \lra (\vv_{u}^{2})_{x} \lra H^{1}(F; \vv_{u})_{x}
\end{equation}
By exactness, the first map $\vv_{u}^{x} \lra (\vv_{u}^{2})^{x}$ induced by $i: \vv_{u} \lra \vv_{u}^{2}$ in Equation \ref{exactxv2} is injective.  However, if this map is an isomorphism, then we have $H^{1}(F; \vv_{u})^{x} \lra (\vv_{u})_{x}$ is an isomorphism.  We now prove this in the following lemma.  
\begin{lemma}\label{bothiso}
Let $u \in U$ be a point in hyperbolic Dehn filling space such $\dim H^{0}(x; \vv_{u}) = 1$ and the map $i: \vv_{u}^{x} \lra (\vv_{u}^{2})^{x}$ in Equation \ref{sesv2} is an isomorphism.  Then $\dim H^{1}(M; \vv_{u}) = 1$ and the induced map $H^{1}(M; \vv_{u}) \lra H^{1}(x; \vv_{u})$ is an isomorphism.  
\end{lemma}
\begin{proof}
Assume that $i: \vv_{u}^{x} \lra (\vv_{u}^{2})^{x}$ is an isomorphism.  Because $u \in U$, $H^{1}(F; \vv_{u})^{x} \neq 0$ by the `Half-Lives Half-Dies Theorem.'  Moreover, $H^{1}(F; \vv_{u})^{x} \lra \vv_{x}$ in Equation \ref{exactxv2} is injective by exactness.  Because $\dim (\vv_{u})_{x} = \dim \vv_{u}^{x} = \dim H^{0}(x; \vv_{u}) = 1$, this means $H^{1}(F;\vv_{u})^{x} \lra (\vv_{u})_{x}$ is an isomorphism.  Note we also have a natural isomorphism $(\vv_{u})_{x} \lra H^{1}(x; \vv_{u})$ which sends $v + \im(1-x)$ to the cohomology class of the crossed homomorphism $f: \pi_{1}(x) \lra \vv$ via $f(x) = v$.  Composing these two isomorphisms yields an isomorphism from $H^{1}(F; \vv_{u})^{x} \lra H^{1}(x; \vv_{u})$.  Finally, we pre-compose this isomorphism with the one given in Equation \ref{RedSeq} to yield $H^{1}(M; \vv_{u}) \lra H^{1}(x; \vv_{u})$.  \\
\\
We claim this map is equal to the restriction map $\res_{x}: H^{1}(M; \vv_{u}) \lra H^{1}(x; \vv_{u})$.  To see this we inspect the definition of $H^{1}(F; \vv_{u}) \lra \vv_{x}$.  Running through the compositions, we take a cohomology class represented by the cocycle $f: \pi_{1}(M) \lra \vv$.  We restrict this to the cochomology class in $H^{1}(F; \vv_{u})$ represented by the cocycle $f: \pi_{1}(F) \lra \vv$.  By construction, this class is fixed under the $x$-action on $H^{1}(F; \vv_{u})$.  \\
\\ 
By definition, the connecting map $H^{1}(F; \vv_{u})^{x} \lra (\vv_{u})_{x}$ takes the cohomology class represented by $f$ to the vector $w + \im(1-x) \in (\vv_{u})_{x}$ such that $(1-x)f = i(w)$ where $i: \vv_{u} \lra \vv_{u}^{2}$.  By construction of the connecting map, the vector $w \in \vv$ is well-defined up to a choice of $\im(1-x)$.  Let us observe that $f$ is determined by its action on both $a$ and $b$, and by definition,
\begin{equation}
(xf)(a) = xf(x^{-1}ax) = f(a) + (a-1)f(x) \text{ and } (xf)(b) = f(b) + (b-1)f(x)
\end{equation}
as $f$ is extendible from $\pi_{1}(F)$ to $\pi_{1}(M)$ by construction.  Thus $(1-x)f = i(f(x))$, and thus the connecting map takes $f \in H^{1}(F; \vv_{u})$ to $f(x) + \im(1-x) \in \vv_{u}$.  Consequently, the composition is equal to $\res_{x}\de H^{1}(M; \vv_{u}) \lra H^{1}(x; \vv_{u})$ as claimed, and thus, we see that if $\dim H^{1}(M; \vv_{u}) = 1$ and the map $\vv_{u}^{x} \lra (\vv_{u}^{2})^{x}$ is an isomorphism, then $\res_{x}: H^{1}(M; \vv_{u}) \lra H^{1}(x; \vv_{u})$ is an isomorphism, and in particular, $\dim H^{1}(M; \vv_{u}) = 1$.  
\end{proof}
We now provide a sufficient condition to guarantee the map $\vv^{x} \lra (\vv^{2})^{x}$ is an isomorphism.  
\begin{lemma}\label{isov2}
Let $u \in U$ be a point in hyperbolic Dehn filling space such $\dim H^{0}(x; \vv_{u}) = 1$.  We have that $\dim (\vv_{u}^{2})^{x} = \dim \ker (1+a+b^{-1}-xb^{-1}a-x^{-1}ba) = \dim \ker \frac{\del w}{\del y}$ where $\frac{\del w}{\del y}$ is the Fox-derivative of the word $w = axa^{-1}y^{-1}$, where $a = xy^{-1}x^{-1}y$, with respect to $y$ as in Equation \ref{wirt}.  Assume further that $\dim \ker \frac{\del w}{\del y} = 1$, then $\vv^{x} \lra (\vv^{2})^{x}$ is an isomorphism.  
\end{lemma}
\begin{proof}
Here we abusively denote the action of $\Ad \rho_{u}(g)$ on $\vv_{u}$ by $g$.  First note that $w = axa^{-1}y^{-1} = (xy^{-1}x^{-1}y)x(y^{-1}xyx^{-1})y^{-1}$.  A straight forward calculation of the Fox-derivative of $w$ with respect to $y$ yields
\begin{equation*}
\frac{\del w}{\del y} = -xy^{-1} - yxy^{-1}x^{-1} - 1 + yxy^{-1} + xy^{-1}x^{-1} 
\end{equation*}
As before, we identify a crossed homomorphism $f: \pi_{1}(F) \lra \vv$ with pairs of vectors $(f(a),f(b))$ in $\vv_{u}^{2}$.  The $x$-action on $\vv_{u}^{2}$ was calculated in Equation \ref{inducedaction}.  The equality for a fixed point becomes
\begin{equation*}
(xb^{-1}+xb^{-1}a)f(a) - xb^{-1}f(b) = f(a) \text{ and } -xa^{-1}f(a) + xa^{-1}f(b) = f(b)
\end{equation*}
which reduces to 
\begin{equation*}
f(a) = (1-ax^{-1})f(b) \text{ and } (1+a+b^{-1}-xb^{-1}a-x^{-1}ba)f(b) = 0
\end{equation*}
Thus, $\dim (\vv^{2})^{x} = \dim \ker (1+a+b^{-1}-xb^{-1}a-x^{-1}ba)$.  If we pre-compose $1+a+b^{-1}-xb^{-1}a-x^{-1}ba$ by $-b$ we are left with
\begin{align}\label{eqdwdy}
&-b-ab-1+xb^{-1}ab+x^{-1}bab  \nonumber\\
&-xy^{-1} - yxy^{-1}x^{-1} - 1 + yxy^{-1} + xy^{-1}x^{-1} = \frac{\del w}{\del y}
\end{align}
as pre-composition by an invertible linear map doesn't affect the dimension of the kernel, we have $\dim (\vv_{u}^{2})^{x} = \dim \frac{\del w}{\del y}$.  \\
\\
Finally, $i: \vv_{u}^{x} \lra (\vv_{u}^{2})^{x}$ is injective by exactness of Equation \ref{exactxv2}.  By hypothesis $\dim H^{0}(x; \vv_{u}) = \vv_{u}^{x} = 1$, so if $\dim \frac{\del w}{\del y} = 1$, then $i: \vv_{u}^{x} \lra (\vv_{u}^{2})^{x}$ is an isomorphism.  
\end{proof}
\end{subsection}

\begin{subsection}{Cohomology Normal Form}\label{secsymset}
It is the purpose of this section to collect results from the previous subsections and combine them into a collection of computationally verifiable conditions to formally certify infinitesimal rigidity of many $(p,q)$-Dehn fillings on the figure-eight knot complement.  By Lemma \ref{suffcond} it suffices to verify four conditions to validate infinitesimal projective rigidity of the Dehn-surgered $(p,q)$-figure eight knot.  
\begin{enumerate}
\item[1.]  Verify $\dim H^{1}(x; \vv_{u}) = 1$
\item[2.]  Verify $\dim H^{1}(M; \vv_{u}) = 1$.  
\item[3.]  Verify $\res_{x}: H^{1}(M; \vv_{u}) \lra H^{1}(x; \vv_{u})$ and $\res_{l}: H^{1}(M; \vv_{u}) \lra H^{1}(l; \vv_{u})$ are both isomorphisms.
\item[4.]  Verify the slope $s_{u}$ of the representation, as in Definition \ref{defslope}, is not equal to $-q/p$.  
\end{enumerate}
As $(p,q)$ are relatively prime and non-zero, the condition that $H^{0}(x; \vv_{u})$ is satisfied as $\rho_{u}(x)$ is non-trivially loxodromic with non-zero translation and rotation so Lemma \ref{circdim} guarantees $\dim \vv_{u}^{x} = 1$.  The condition that $\dim H^{1}(M; \vv_{u}) = 1$ can be verified by use of Lemma \ref{bothiso} and Lemma \ref{isov2}.  Guaranteeing that $\res_{x}: H^{1}(M; \vv_{u}) \lra H^{1}(x; \vv_{u})$ is an isomorphism can similarly be verified by the use of these lemmas.  The condition that $\res_{l}: H^{1}(M; \vv_{u}) \lra H^{1}(l; \vv_{u})$ is an isomorphism can be verified by calculating a formally certified bound on the slope of the representation.  Provided the slope is non-zero and lies sufficiently far away from $-q/p$, then we may guarantee rigidity.  \\
\\
Before proceeding, let us recall we identify the cocycles $Z^{1}(M; \vv_{u})$ with the kernel of the map $\vv \times \vv \lra \vv$ that takes a pair of vectors $(v,w)$ to $\frac{\del w}{\del x}v + \frac{\del w}{\del y}w$ where $\frac{\del w}{\del x}$ and $\frac{\del w}{\del y}$ are the Fox-derivatives of the word 
\begin{equation}\label{definingword}
w = axa^{-1}y^{-1} = xy^{-1}x^{-1}yxy^{-1}xyx^{-1}y^{-1}
\end{equation}
This identification comes from the association between a crossed homomorphism $f: \pi_{1}(M) \lra \vv$ and sending $f$ to its evaluation on $x$ and $y$, $(f(x),f(y)) \in \vv\times \vv$.  The coboundaries $B^{1}(M; \vv_{u})$ are identified with the image of the map $i: \vv \lra \vv\times \vv$ taking a vector $v \in \vv$ to $\left((1-x)v,(1-y)v\right)$.  The expository article by Goldman illustrates this association in thorough detail \cite{Goldman2020Parallelism}.  \\
\\
If we are able to verify $\dim \ker \frac{\del w}{\del y} =1 $, then by Lemma \ref{isov2} and Lemma \ref{bothiso}, $\dim H^{1}(M; \vv_{u}) = 1$ and $\res_{x}: H^{1}(M; \vv_{u}) \lra H^{1}(x; \vv_{u})$ is an isomorphism.  We may use this to then find a normal form for a generator of $H^{1}(M; \vv_{u})$. \\
\\
In the case that $\dim \ker \frac{\del w}{\del y} = 1$, then $\ker \frac{\del w}{\del y}$ is generated by the vector $(1-y)a_{u}$ where $a_{u} \in \vv$ is smooth choice of a generator for $H^{0}(T^{2}; \vv_{u})$.  The vector $(1-y)a_{u}$ generates $\ker \frac{\del w}{\del y}$ because
\[
 \frac{\del w}{\del x}(1-x)a_{u} +  \frac{\del w}{\del y}(1-y)a_{u} =   \frac{\del w}{\del y}(1-y)a_{u} = 0
\]
Because $\dim \ker \frac{\del w}{\del y} = 1$, this means we may find a representative of the cohomology $[(v,m)] \in H^{1}(M; \vv_{u}) =  Z^{1}(M; \vv_{u})/B^{1}(M; \vv_{u})$ with $v \neq 0$.  Because $B^{1}(M; \vv_{u})$ is 9-dimensional, and $H^{1}(M; \vv_{u})$ is 1-dimensional, the cocycles $Z^{1}(M; \vv_{u})$ must be 10-dimensional.  The intersection $Z^{1}(M; \vv_{u}) \cap 0\times \vv = 0 \times \ker \frac{\del w}{\del y}$ consists entirely of coboundaries and thus, every element in $H^{1}(M; \vv_{u})$ of the form $[(0,m)]$ is trivial.  \\
 \\
However, if we choose a cohomology representative $(v,m) \in Z^{1}(M; \vv_{u})$ with $v \neq 0$, we may decompose $v = sa_{u} + (1-x)u$ for some non-zero scalar $s \in \R$ and $u \in \vv$.  Subtracting the coboundary part $(1-x)u$, and scaling out by $s \neq 0$, we may normalize our cohomology representative to be of the form $[(a_{u},m)] \in H^{1}(M; \vv_{u})$.  Let us denote the map on 1-chains induced by $[(a_{u}, m)]$ by $z_{u}$.  \\
\\
Since $z_{u}(x) = a_{u}$, Equation \ref{crossedcom} yields that $z_{u}(l) = ta_{u}$ for some $t \in \R$.  Provided $t \neq 0$, this would imply the induced map on $H^{1}(M; \vv_{u}) \lra H^{1}(l; \vv_{u})$ is an isomorphism as $\res_{l}: H^{1}(M; \vv_{u}) \lra H^{1}(l ; \vv_{u})$ will take $[(a_{u},m)]$ to $ta_{u}$.  Moreover if $t\neq 0$, then $t = s_{u}^{-1}$ where $s_{u}$ is the slope of the representation $\rho_{u}\de \pi_{1}(M) \lra \PSL(2,\C)$ as defined in Definition \ref{defslope}.  So long as $s_{u} \neq -q/p$, then Conditions $(1), (2), (3)$ and $(4)$ as stated in the beginning of this subsection will be verified.  We summarize the results of this analysis in the following theorem.  
\begin{theorem}\label{thmrigid}
Let $u \in U$ be a point in hyperbolic Dehn filling space for which $\dim \frac{\del w}{\del y} = 1$ and $\dim H^{0}(x; \vv_{u}) = 1$.  Then $\dim H^{1}(M; \vv_{u}) = 1$ and $\res_{x}\de H^{1}(M; \vv_{u}) \lra H^{1}(x; \vv_{u})$ is an isomorphism.  Additionally, there exists a be a smooth local choice of basis elements $a_{u}$ in $H^{0}(T^{2}; \vv_{u})$ and a cohomology class $z_{u} \in H^{1}(M; \vv_{u})$ for which $z_{u}(x) = a_{u}$ and $z_{u}(l) = ta_{u}$ where $t$ is some real number.  If $t$ is non-zero, then $t = s_{u}^{-1}$ where $s_{u}$ is the slope of the representation $\pi_{1}(M) \lra \PSL(2,\C)$.  In this case, if $s_{u} \neq -q/p$, then the Dehn-surgered manifold $M_{p,q}$ is infinitesimally projectively rigid, and $H^{1}(M_{p,q}, \vv_{u}) = 0$.   
\end{theorem}
In the next section we employ machinery to create certified numerical estimates on the slope $s_{u}$ to formally verify many Dehn-surgeries are infinitesimally projectively rigid.  
\end{subsection}
\end{section}

\begin{section}{Interval Analysis}\label{ia}
In this section we provide experimentally verified calculations to show that the cohomology groups $H^{1}(M_{p,q}; \vv_{u}) = 0$ are trivial for many relatively prime $p,q$.  Our calculations will implement the use of Matlab and the Interval Arithmetic package INTLAB by Rump \cite{Ru99a}, however informal calculations can be found in the Mathematica file \texttt{PRFE\textunderscore Mathematica.nb}.  The calculations carried out in this section are all done relative to the basis $\beta = \{v_{1}, \hdots, v_{9}\}$ provided in Equation \ref{vbasis}.  
\begin{subsection}{Preliminaries}\label{intprelim}
Here we summarize some basic properties regarding interval arithmetic.  We direct the interested reader to Moore, Kearfott, and Cloud's Interval Analysis text for more thorough details \cite{Moore2009Introduction}.  \\
\\
For the purpose of our context, let us restrict temporarily to the real numbers.  The typical closed interval in $\R$ is denoted $[a,b]$ for $a \leq b$ and is the set of all $x \in \R$ satisfying $a \leq x \leq b$.  We say an interval is degenerate if $a = b$.  Let us say we have some quantities $x,y$ that we cannot say precisely what they are, but we may numerically approximate them and say, with certainty, they lie in some intervals $[a,b]$ and $[c,d]$ respectively.  Then we can say, with certainty, that 
\begin{equation*}
a \leq x \leq b \text{ and } c \leq y \leq d \implies a + c \leq x + y \leq b + d
\end{equation*}
Similarly, we can say that if $y \in [c,d]$ then $-y \in [-d,-c]$, so that $a-d \leq x - y \leq b - c$.  This gives us a notion of addition and subtraction on intervals in this way.  Specifically we define
\begin{equation*}
[a,b] + [c,d] := [a+c,b+d] \text{ and } [a,b] - [c,d] := [a,b] + \left(-[c,d]\right) = [a-d,b-c]
\end{equation*}
Multiplication on intervals is defined by
 \begin{equation*}
[a,b] \cdot [c,d] := \left[ \min{S}, \max{S} \right] \text{ where }S = \{ac, ad, bc,bd\}
\end{equation*}
One can go through all possible cases of signs on $a,b,c,d$ to verify this definition guarantees that if $x \in [a,b]$ and $y \in [c,d]$, then $xy \in [a,b]\cdot [c,d]$.  We may also define the inverse of an interval that does not contain $0$ in the straight forward manner
\begin{equation*}
1/[a,b] := \left[1/b, 1/a\right]
\end{equation*}
and then division on intervals is defined by
\begin{equation*}
[a,b]/[c,d] := [a,b]\cdot \left(1/[c,d]\right)
\end{equation*}
The operations of addition and multiplication turn the set of all intervals into a semi-ring and constitute the most basic operations of \emph{interval arithemetic} which is an indispensable tool to certify zeros of a function lie inside an interval.  Methods such as the Interval Newton Method or Krawczyk's Method can be used verify whether or not a zero of a function is contained within an interval \cite[Chapter 8]{Moore2009Introduction}.  In fact, the latter is how SnapPy verifies the intervals for shapes of ideal tetrahedra \cite{SNAPKrawczyk}.  \\
\\
The operations above can be extended to the complex numbers.  In this situation the intervals are replaced with either boxes or disks depending on which convention is adopted.   If boxes are used, multiplication and division are not so easily defined whereas addition and subtraction are.  On the other hand, if disks are used, multiplication and division are easily defined, whereas addition and subtraction are more difficult.  As it so happens, SnapPy uses an interval package that adopts the rectangular boxes so this method will be reflected in our later estimates.\\
\\
We conclude this subsection with a brief review of \emph{interval matrix algebra}.  By an \emph{interval matrix} $\bf{A}$, we mean an $m \times n$-array of $a_{ij}$ where each $a_{ij}$ is an interval.  \emph{Interval vectors} are defined similarly.  With these two definitions we may define an interval linear system $\bf{A}x = \bf{b}$ where $\bf{A}$ is an $n \times n$ interval matrix and $\bf{b}$ is an $n\times 1$ interval vector.  There are several inequivalent notions of the solution set $\bf{x}$ of an interval matrix system.  We summarize the three popular ones as defined in \cite{Shary1995On}.
\begin{enumerate}
\item The \emph{tolerable solution set} is the set of all solutions such that $Ax$ lands inside of $\bf{b}$ for any choice of $A \in \bf{A}$.  It is denoted 
\[
\sum_{\forall\exists}\left(\bf{A},\bf{b}\right) = \left\{ x \in \R^{n} \, | \, \left(\forall A \in \bf{A}\right)\left(\exists b \in \bf{b} \right)\left(Ax = b\right) \right\}
\]
\item The \emph{controllable solution set} which is the set of all vectors $x \in \R^{n}$ such that for any $b \in \bf{b}$ we can find an $A \in \bf{A}$ so that $Ax = b$.  It is denoted
\[
\sum_{\exists\forall}\left(\bf{A},\bf{b}\right) = \left\{ x \in \R^{n} \, | \, \left(\forall b \in \bf{b}\right)\left(\exists A \in \bf{A} \right)\left(Ax = b\right) \right\}
\]
\item Finally, the \emph{united solution set} is the solution set of all possible solutions $Ax = b$ with $A \in \bf{A}$ and $b\in \bf{b}$.  It is denoted
\[
\sum_{\exists\exists}\left(\bf{A},\bf{b}\right) = \left\{ x \in \R^{n} \, | \, \left(\exists A \in \bf{A}\right)\left(\exists b \in \bf{b} \right)\left(Ax = b\right) \right\}
\]
\end{enumerate}

The united solution set is far and away the most studied and well understood solution set.  As a preliminary observation, note that if there exists a singular matrix $A \in \bf{A}$, and a $b \in \bf{b}$ for which there exists a solution to $Ax =b$, then the united solution set is unbounded.  This behavior can be avoided provided a non-singularity condition be imposed on the interval matrix $\bf{A}$.  
\begin{definition}\label{regular}
An $n\times n$ interval matrix $\bf{A}$ is called \emph{regular}, if each $A \in \bf{A}$ is non-singular.  
\end{definition}
Under these conditions, we have that the united solution set to $\bf{A}x = \bf{b}$ is a bounded subset of $\R^{n}$.  While INTLAB typically does not calculate the exact united solution set, it provides enclosures to it via the \texttt{verifylss} command.   
\end{subsection}

\begin{subsection}{Interval Analysis Algorithm}\label{iaa}
Recall the gluing diagram in Figure \ref{figgluingdiag}.  The shapes of the ideal polyhedra $z_{1}$ and $z_{2}$ satisfy the gluing equations in Equations \ref{snapglue1} and \ref{snapglue2}.  Let $p$ and $q$ be relatively prime integers between $1$ and $60$ such that $p/q$ is not an exceptional slope, i.e. $(p,q) \notin \{(1,1),(2,1),(3,1),(4,1)\}$.  Here we may inspect strictly positive $p$ and $q$, as $M_{-p,-q}$ is homeomorphic to $M_{p,q}$.  Because $M$ is amphicheiral, we also have that $M_{p,q}$ is homeomorphic to $M_{-p,q}$, so we restrict our attention to $p$ and $q$ that are strictly positive.  \\
\\
Let $z_{1}$ and $z_{2}$ be chosen so that $pu + qv - 2\pi i = 0$ where $u = \Log\left(z_{1}(1-z_{2})\right)$ and $v = 2\Log\left(-(1-z_{2})^{2}/z_{2}\right)$ as in Equation \ref{complexlen}.  Through the use of SnapPy we may approximate $z_{1}$ and $z_{2}$ to high precision.  The files \texttt{SNAPPY\textunderscore Shapes\textunderscore Mathematica.nb} and \texttt{Snappy\textunderscore Shapes.m} both contain the data of all such shapes $(z_{1},z_{2})$ split into their real and imaginary parts certifiably up to 16 digits of precision for each pair of $(p,q)$ with $1\leq p,q \leq 60$ relatively prime and non-exceptional.   It is worth mentioning, even though the equation $pu + qv - 2\pi i = 0$ is non-linear in $z_{1}$ and $z_{2}$, Krawczyk's Method or Newton's are highly effective in this situation as the derivatives in $z_{1}$ and $z_{2}$ of $u$ and $v$ are \emph{rational} functions of $z_{1}$ and $z_{2}$, and thus very well suited to for the use of interval arithmetic.  \\
\\
With these preliminary notions established, we may apply them to our situation to computationally verify that these manifolds $M_{p,q}$ are projectively rigid.  Let $\rho_{u}\de \pi_{1}(M) \lra \PSL(2,\C)$ denote the representation associated to the $(p,q)$ Dehn filling of $M$.  Because neither $p$ nor $q$ is zero, $u$ has non-trivial translation and rotation, thus $\dim H^{0}(x; \vv_{u}) = 1$.  By Theorem \ref{thmrigid}, it suffices to verify two conditions.  
\begin{enumerate}\label{eqsimplecond}
\item[1.] Verify $\frac{\del w}{\del y}$ is rank 8.  
\item[2.] Verify the slope of the representation is non-zero and not equal to $-q/p$. 
\end{enumerate}
These two statements are rigorously verified via the function \texttt{rigcheck.m} whose input is a pair of relatively prime $(p,q)$ where $1 \leq p,q \leq 60$ such that $p/q$ is a non-exceptional slope.  \texttt{rigcheck.m} references the list \texttt{Snappy\underscore Shapes.mat} which associates to each $(p,q)$ four real intervals, $\bf{z_{1x}},\bf{z_{1y}},\bf{z_{2x}}$ and $\bf{z_{2y}}$.  Each interval contains the real and imaginary parts of $z_{1}$ and $z_{2}$ as defined in Figure \ref{figgluingdiag} and the equation $pu + qv = 2\pi i$ where $u$ and $v$ are defined by Equation \ref{complexlen}.  These shapes are accurate up to 16 digits of precision and obtained via SnapPy.  The file \texttt{SNAPPY\underscore instructions.rtf} includes instructions on how to obtain these shapes within SnapPy.  \\
\\
With each such $(p,q)$, we calculate interval matrices via Equation \ref{eqadact} that are guaranteed to contain $\Ad \rho_{u}(x)$ and $\Ad \rho_{u}(y)$ where $\rho_{u}: \pi_{1}(M) \lra \PSO(3,1)$ is the representation corresponding to the $(p,q)$ Dehn-surgery of $M$.  The fact that $\Ad \rho_{u}(x)$ and $\Ad \rho_{u}(y)$ are contained in their interval counterparts $\bf{Ad} \bf{\rho}_{u}(x)$ and $\bf{Ad} \bf{\rho}_{u}(y)$ is a consequence of our representation $\rho_{u}$ being \emph{rational} in the real and imaginary parts of $z_{1}$ and $z_{2}$.  All the calculations in the entries of $\Ad \rho_{u}(x)$ and $\Ad \rho_{u}(y)$ are rational expressions of $z_{1x}, z_{1y}, z_{2x}$ and $z_{2y}$.  By the operations of interval arithmetic summarized in Section \ref{intprelim}, $\Ad \rho_{u}(x)$ and $\Ad \rho_{u}(y)$ must therefore by contained in $\bf{Ad} \bf{\rho}_{u}(x)$ and $\bf{Ad} \bf{\rho}_{u}(y)$.  To ease notation, we abusively denote the matrices $\Ad \rho_{u}(x)$ and $\Ad \rho_{u}(y)$ by $x$ and $y$ and their interval counterparts by $\bf{x}$ and $\bf{y}$.  \\
\\
To verify that $\frac{\del w}{\del y}$ is rank 8, it suffices to find a rank 8 sub-matrix.  By Equation \ref{eqdwdy}, $\frac{\del w}{\del y}$ is also a matrix that is some rational expression in the real and imaginary parts of $z_{1}$ and $z_{2}$.  We calculate interval matrix containing $\frac{\del w}{\del y}$ by calculating each summand of Equation \ref{eqdwdy}, and then summing those together.  Each summand may be found in the \texttt{Symbolic\underscore Matlab \underscore Matrices} folder.  For example, the Matlab function $\texttt{YXinvY.m}$ inputs the intervals $\bf{z_{1x}}, \bf{z_{1y}},\bf{z_{2x}}$ and $\bf{z_{2y}}$, and returns an interval matrix that certifiably contains the matrix $yx^{-1}y$.  We apply the same process to calculate certified interval bounds on the quantities $\bf{\frac{\del w}{\del x} }$, $\bf{\frac{\del l}{\del x} }$, $\bf{\frac{\del l}{\del y} }$, and $\bf{a_{u}}$ as well. \\
\\
Verifying that $\frac{\del w}{\del y}$ is rank 8 can be done by verifying that $\bf{\frac{\del w}{\del y} }$ contains an $8\times 8$-regular submatrix.  For simplicity, we simply choose the top-left $8\times 8$-submatrix of $\bf{\frac{\del w}{\del y} }$ and verify regularity via INTLAB's \texttt{isregular()} command.  In \texttt{rigcheck.m}, if the regularity check fails to pass, then the function ceases as Theorem \ref{thmrigid} can only be applied in the situation where $\frac{\del w}{\del y}$ is rank 8 which guarantees that both $H^{1}(M; \vv_{u})$ is dimension 1 and $\res_{x}: H^{1}(M; \vv_{u}) \lra H^{1}(x; \vv_{u})$ is an isomorphism.\\
\\
Once the condition that $\frac{\del w}{\del y}$ is rank 8 has been verified, \texttt{rigcheck.m} proceeds to calculating a certified bound on the slope of the representation.  We proceed by calculating the normal form for the cohomology class given by $[(a_{u},m)] \in H^{1}(M; \vv_{u})$ as in Section \ref{secsymset}.  We know that $a_{u}$ and $m$ satisfy the equation 
\begin{equation}\label{cocyclecond}  
\frac{\del w}{\del x}a_{u} + \frac{\del w}{\del y}m = 0
\end{equation}
So in particular $\frac{\del w}{\del y}m = -\frac{\del w}{\del x}a_{u}$.   The difficulty Equation \ref{cocyclecond} is that the matrix $\frac{\del w}{\del y}$ is singular.  Its kernel is generated by $(1-y)a_{u}$, and thus the solution set to Equation \ref{cocyclecond} is an unbounded set.  To circumvent this difficulty, we restrict to the solution set of $\frac{\del w}{\del y}x = -\frac{\del w}{\del x} a_{u}$ by intersecting it with the hyperplane $x_{9} = -1$ where here we are representing both vectors and matrices relative to the basis in Equation \ref{vbasis}.  As $(1-y)a_{u}$ is not parallel to the hyperplane $x_{9} = -1$, we may always guarantee a \emph{unique} solution to $\frac{\del w}{\del y}x = -\frac{\del w}{\del x} a_{u}$ with $x_{9} = -1$ as $\frac{\del w}{\del y}$ is rank 8.  Let $\frac{\del w}{\del y}_{8\times 8}$ denote the top-left $8\times 8$ submatrix of $\frac{\del w}{\del y}$.  We now consider the following equation
\begin{equation}\label{newcocyc}
\frac{\del w}{\del y}_{8\times 8} m_{8\times 1} = -\left(\frac{\del w}{\del x}a_{u} \right)_{8\times 1} + \frac{\del w}{\del y}_{*\times 9}
\end{equation}
where $\frac{\del w}{\del y}_{*\times 9}$ denotes the first $8$ rows of the $9$-th column of $\frac{\del w}{\del y}$.  The solution to Equation \ref{newcocyc} will be the first 8-coordinates of the unique solution to Equation \ref{cocyclecond} where $x_{9} = -1$ which we denote by $m_{8\times 1}$.  \\
\\
To obtain a certified bound on $m_{8\times 1}$, we consider the corresponding \emph{interval linear system},
\begin{equation}\label{int_linear}
\bf{\frac{\del w}{\del y} }_{8\times 8} \bf{m}_{8\times 1} = -\left(\bf{\frac{\del w}{\del x}a_{u}} \right)_{8\times 1} + \bf{\frac{\del w}{\del y}_{*\times 9}}
\end{equation}
Using INTLAB's \texttt{verifylss.(\bf{A},\bf{b})}, we obtain a certified bound $\bf{m}_{8\times 1}$ containing $m_{8\times 1}$.  Because the interval matrix $\bf{\frac{\del w}{\del y} }_{8\times 8}$ is regular, we are guaranteed a finite enclosure of this equation.  Padding the solution interval vector $\bf{m}_{8\times 1}$ at the end with the degenerate interval $[-1,-1]$ yields a certified bound on a solution Equation \ref{cocyclecond}.  In \texttt{rigcheck.m} we denote this solution by \texttt{m}.  \\
\\
With a verified enclosure of $[(a_{u},m)] \in H^{1}(M; \vv_{u})$, we may calculate a verified enclosure of the image of the map $\res_{l}\de H^{1}(M; \vv_{u}) \lra H^{1}(l; \vv_{u})$.  By Theorem \ref{thmrigid}, $z_{u}(l) = ta_{u}$ where if $t$ is non-zero, then we guarantee $\res_{l}$ is an isomorphism and $t = s_{u}^{-1}$.  By the properties of crossed homomorphisms,
\begin{equation}\label{zul}
z_{u}(l) =  \frac{\del l}{\del x} z_{u}(x) + \frac{\del l}{\del y} z_{u}(y)  = \frac{\del l}{\del x} a_{u} + \frac{\del l}{\del y} m = ta_{u} 
\end{equation}
which is contained in the corresponding interval vector $\bf{\frac{\del l}{\del x} a_{u} + \frac{\del l}{\del y} m}$.  We denote this enclosure by $\bf{z_{u}(l)}$.  The Fox-derivatives of the word $l = b^{-1}a^{-1}ba = (yx^{-1}y^{-1}x)(xy^{-1}x^{-1}y)$ are calculated below.
\begin{align*}\label{foxforl}
\frac{\del l}{\del x} &= -yx^{-1}+yx^{-1}y^{-1}+yx^{-1}y^{-1}x-yx^{-1}y^{-1}xxy^{-1}x^{-1}\\
\frac{\del l}{\del y} &=1 - yx^{-1}y^{-1}-yx^{-1}y^{-1}xxy^{-1}+yx^{-1}y^{-1}xxy^{-1}x^{-1}
\end{align*}
We take $\bf{z_{u}(l)}$ and compare it to $\bf{a_{u}}$ via the function \texttt{slopebound.m}.  Because $\bf{a_{u}}$'s 4th entry is the degenerate interval $[0,0]$, see \texttt{au} in \texttt{PRFE\textunderscore Mathematica.nb} for this calculation, we discard the 4th entry entirely.  We then divide each entry of $\bf{a_{u}}$ by $\bf{z_{u}(l)}$.  If even a single interval of the quotient manages to avoid zero entirely, we guarantee the slope of the representation $s_{u}$ is non-zero and well-defined.  In this case, each such entry interval of the quotient, provided the entry is well-defined, is guaranteed to contain the true slope of the representation, $s_{u}$.  So if in addition, even a single interval of the quotient manages to avoid $-q/p$, then by Theorem \ref{thmrigid}, we will have guaranteed the manifold $M_{p,q}$ is projectively rigid.  \\
\\
The function \texttt{slopebound.m} looks through each entry interval of the quotient of $\bf{a_{u}}$ by $\bf{z_{u}(l)}$, discards the ones that are infinite or contain zero, if any, and the intersects over all the remaining intervals.  Provided there is at least one such interval, this intersection is guaranteed to contain the slope of our representation $s_{u}$.  We denote the quantity by $\bf{s_{u}}$.  \texttt{rigcheck.m} checks that both $0$ and $-q/p$ are not in this interval, and if so, declares the manifold $M_{p,q}$ as projectively rigid.  
\end{subsection}

\begin{subsection}{$(2,3)$ Dehn-Surgery}\label{23surgery}
In this section we briefly run through an example of the above calculations with the $(2,3)$ Dehn-surgery on $M$.  One can test and verify the results here in \texttt{sandbox.m} or run \texttt{rigcheck.m} with $(2,3)$ as the input.  The SnapPy certified shapes of the ideal tetrahedra are given by
\begin{align*}
z_{1x} &= 0.567343784636165\hdots  \\
z_{1y} &= 0.442016550101567\hdots  \\
z_{2x} &= 0.322271789512312\hdots \\
z_{2y} &= 0.449431980686431\hdots
\end{align*}
The trailing $\hdots$ indicates that these calculations are correct \emph{including} the last digit.  So for example the last 5 in $z_{1x}$ is certifiably correct, however we do not know what the next digit is.  For our purposes, we used SnapPy to calculate these to 15 decimals places of precision.  We chose our $\epsilon = 10^{-15}$, so that for example
\begin{align}
z_{1x} \in&\,  [0.567343784636165-\epsilon,0.567343784636165 + \epsilon]  \nonumber\\
&= [0.567343784636164, 0.567343784636166]
\end{align}
We denote the intervals of the real and imaginary parts of $z_{1}$ and $z_{2}$ padded by $\epsilon$ by their boldface counterparts.  We calculate the interval matrix $\bf{\frac{\del w}{\del y}}$ and perform a regularity check via \texttt{isregular()} on the top left $8\times 8$ sub-matrix.  INTLAB indicates this interval matrix is indeed regular, and thus, the rank of $\frac{\del w}{\del y}$ is rank 8.  Using the method described in Section \ref{iaa} and Equation \ref{zul} we calculate $\bf{a_{u}}$ and $\bf{z_{u}}(l)$ which are 
\begin{scriptsize}
\begin{equation}\label{23au}
\bf{a_{u}} = \left(
\begin{array}{c}
\left[  -0.35148464890333,  -0.35148464890329\right] \\
\left[  -0.35148464890333,  -0.35148464890329\right] \\
\left[  -0.64851535109671,  -0.64851535109667\right] \\
\left[   0.00000000000000,   0.00000000000000\right] \\
\left[   0.83367747699777,   0.83367747699779\right] \\
\left[  -0.83367747699779,  -0.83367747699777\right] \\
\left[   0.08916928929421,   0.08916928929423\right] \\
\left[  -0.08916928929423,  -0.08916928929421\right] \\
\left[   1.00000000000000,   1.00000000000000\right] 
\end{array}
\right) \text{, } \bf{z_{u}}(l) = 
\left(
\begin{array}{c}
\left[   3.85413503155950,   3.85413549646991\right] \\
\left[   3.85413501053848,   3.85413551749158\right] \\
\left[   7.11116655098591,   7.11116658200069\right] \\
\left[  -0.00000024435903,   0.00000024435797\right] \\
\left[  -9.14152522655167,  -9.14152510260290\right] \\
\left[   9.14152481827732,   9.14152551087727\right] \\
\left[  -0.97776823755712,  -0.97776810468840\right] \\
\left[   0.97776780689252,   0.97776853535285\right] \\
\left[ -10.96530191882957, -10.96530174218687\right] 
\end{array}
\right)
\end{equation}
\end{scriptsize}
We know the true value of $z_{u}(l)$ has its 4th entry is $0$, so when we take the quotient of the entries, we ignore this particular entry.  This gives us
\begin{equation*}
\bf{a_{u}}/\bf{z_{u}(l)} = \left(
\begin{array}{c}
\left[  -0.09119676555835,  -0.09119675455759\right] \\ 
\left[  -0.09119676605575,  -0.09119675406018\right] \\ 
\left[  -0.09119676025684,  -0.09119675985908\right] \\ 
\left[  -0.09119676067623,  -0.09119675943969\right] \\ 
\left[  -0.09119676351269,  -0.09119675660323\right] \\ 
\left[  -0.09119676625435,  -0.09119675386162\right] \\ 
\left[  -0.09119679402988,  -0.09119672608613\right] \\ 
\left[  -0.09119676079252,  -0.09119675932340\right] 
\end{array} \right)
\end{equation*}
Because the true slope $s_{u}$ is contained in \emph{each} entry, we may take the highest lower bound of the above entries, and the lowest upper bound of the above entries to obtain a certified bound of the slope.  This yields
\begin{equation}\label{23slopebound}
\textbf{s}_{\bf{u}} = [  -0.09119676025684,  -0.09119675985908] 
\end{equation}
Note that the ratio $-3/2 = -1.5$ is not contained in $\textbf{s}_{\bf{u}}$, and thus by Theorem \ref{thmrigid}, we have that the $(2,3)$ Dehn filling of the figure-eight knot complement is projectively rigid.  This calculation is carried out step by step in \texttt{sandbox.m}.  These same calculations are carried out by the function $\texttt{rigcheck.m}$ whose output it a string consisting of
\begin{equation}\label{rigcheckoutput}
(p,q)\phantom{==} b_{1}\phantom{==} b_{2} \phantom{==}-q/p \phantom{==} \textbf{s}_{\bf{u}} 
\end{equation}
where $b_{1}$ and $b_{2}$ are booleans.  $b_{1}$ is true if $\frac{\del w}{\del y}$ is certifiably rank 8 via the regularity check described in Section \ref{iaa} and false otherwise.  $b_{2}$ is true if both $-q/p$ and $0$ are outside $ \textbf{s}_{\bf{u}}$ and false otherwise.  \\
\\
The program \texttt{main.m} performs \texttt{rigcheck.m} for all of the 2199 integer pairs $(p,q)$ that are relatively prime non-exceptional slopes in the range of $1\leq p, q \leq 60$.  The results of each \texttt{rigcheck.m} is put into the \texttt{.txt}-file \texttt{rigresults.txt}.  The results of this program which took approximately 24 hours to run indicate that all such Dehn fillings are projectively rigid.  The only one that could not verify projective rigidity was the `trivial' one $(1,1)$ which is exceptional and artificially given `bad' values to ensure that the function \texttt{rigcheck.m} is performing properly.  \\
\\
It is worth mentioning the calculation for the $(2,3)$ surgery is also informally conducted in \texttt{PRFE\textunderscore Mathematica.nb}.  Towards the end of the notebook, one may change the variables \texttt{dehnp} and \texttt{dehnp} from $(2,3)$ to any pair of the relatively prime $p$ and $q$ we've addressed in this paper to see calculations done at a much quicker, albeit informal, manner.  \\
\\
We conclude by remarking that the combined results of this paper along with Heusener and Porti \cite{Heusener2011Infinitesimal} suggest that it is possible that all such Dehn-surgeries on the figure eight knot complement resulting in an honest manifold, i.e. when $(p,q)$ are relatively prime and non-exceptional, are infinitesimally projectively rigid.  If one could calculate an effective bound on the set $0 \in U$ for which they guarantee infinitesimal projective rigidity, then one could possible check the remaining finite cases with the algorithm provided in this paper.
\end{subsection}
\end{section}

\begin{section}{Bibliography}
\bibliographystyle{plain}
\bibliography{refs}
\end{section}
\end{document}